\documentclass[reqno]{amsart}
\usepackage{amsmath,amssymb,amsthm,amsrefs, amscd}
\usepackage{xcolor}

\usepackage[english]{babel}
\usepackage[utf8]{inputenc}	

\newtheorem{thm}{Theorem}[section]
\newtheorem{pro}[thm]{Proposition}
\newtheorem{lem}[thm]{Lemma}
\newtheorem{cor}[thm]{Corollary}
\theoremstyle{definition}
\newtheorem{defi}[thm]{Definition}
\newtheorem{rem}[thm]{Remark}

\newtheorem{exa}[thm]{Example}

\newcommand{\nd}{\noindent}
\newcommand{\vu}{\vspace{.1cm}}
\newcommand{\vd}{\vspace{.2cm}}

\newcommand{\e}{\mathfrak{e}}

\allowdisplaybreaks 
\begin{document}
	\allowdisplaybreaks

\title{Multiplier Hopf Algebras: Globalization for Partial Actions}

\maketitle

\author{GRAZIELA FONSECA
	\footnote{Instituto Federal de Educa\c c\~ao, Ci\^encia e Tecnologia Sul-Rio-Grandense, Charqueadas, RS  96.745-000, Brazil. E-mail address: grazielalangone@gmail.com.}}, \author{ENEILSON FONTES \footnote{Instituto de Matem\'atica, Estat\'istica e F\'isica, {Universidade Federal do Rio Grande}, Rio Grande, RS 96.201-900, Brazil. E-mail address: eneilsonfontes@furg.br.}},  \author{GRASIELA MARTINI
\footnote{Instituto de Matem\'atica, Estat\'istica e F\'isica, Universidade Federal do Rio Grande, Rio Grande, RS 96.201-900, Brazil. E-mail address: grasiela.martini@furg.br.}}

\begin{abstract}
In partial action theory, a pertinent question is whenever given a partial action of a Hopf algebra $A$ on an algebra $R$, it is possible to construct an enveloping action. The authors Alves and Batista, in \cite{Muniz}, have shown that this is always possible if $R$ is unital. We are interested in investigating the situation where both algebras $A$ and $R$ are not necessarily unitary. A nonunitary natural extension for the concept of Hopf algebras was proposed by Van Daele, in \cite{Multiplier}, which is called multiplier Hopf algebra. Therefore, we will consider partial actions of multipliers Hopf algebras on algebras with a nondegenerate product and we will present a globalization theorem for this structure. Moreover, Dockuchaev, Del Rio and Sim\'on, in \cite{Micha}, have shown when group partial actions on nonunitary algebras are globalizable. Based in this paper, we will establish a bijection between globalizable group partial actions and partial actions of a multiplier Hopf algebra.
\end{abstract}

\vu
{\nd\scriptsize{\bf Keywords:} multiplier Hopf algebra, partial action, globalization\\}
{\nd\scriptsize{\bf Mathematics Subject Classification:} primary 16T99; secondary 20L99}

\section{introduction}

Natural examples of partial actions can be easily obtained by restriction of  global actions to subsets not necessarily invariant by such actions. On the other hand, to investigate the existence of an enveloping action for a determined partial action means to find out under what conditions such a partial action can be obtained, less than an equivalence, as restriction of a global one.

The first theorem in the algebraic context about the existence of enveloping actions is due to R. Exel and M. Dokuchaev, in \cite {exel2}, for partial group actions on algebras. Later, M. Alves and E. Batista extended these ideas to the context of  partial actions of Hopf algebras, in \cite{MunizandBatista}. Precisely, for a partial action of a Hopf algebra  $A$ on a unital algebra $L$, the authors have shown that there is a global action on an algebra $R$ which contains $L$ as an ideal such that the partial action of $A$ on $L$ is induced by the global action of $A$ on $ R $. In this same paper, the uniqueness of enveloping actions has been also investigated.

 Thereby, our motivation is to investigate what occurs to these results in the context of multiplier Hopf algebras, which has been proposed by A. Van Daele in \cite {Multiplier}.

Recalling, an algebra $A$ over a field $\Bbbk$ is endowed with a \emph{nondegenerate} product when it has the property  that $a=0$ if $ab =0$ for all $b \in A$ and $b=0$ if $ab=0$ for all $a \in A$. In this case, we will denote the multiplier algebra of $A$ by $M(A)$ which is the usual $\Bbbk$-vector space of all the ordered pairs $(U,V)$ of linear maps on $A$ such that $ V(a)b=aU(b)$, for all $a,b\in A$. It follows immediatly that $U(ab)=U(a)b$ and $V(ab)=aV(b)$, for all $a,b\in A$.
The product is given by the rule $(U,V)(U',V')=(U\circ U',V'\circ V).$ Such an algebra is associative and unital with identity element given by the pair $1=(\imath,\imath)$ where $\imath$ denotes the identity map of $A$. Moreover, there exists a canonical algebra monomorphism $\jmath:A\to M(A)$ given by $a\mapsto (U_a,V_a)$, where $U_a$ (resp., $V_a$) denotes the left (resp., right) multiplication by $a$, for all $a\in A$. Furthermore, if $A$ is unital, then $\jmath$ is an isomorphism.

\vu

Let $A$ be an algebra with a nondegenerate product. A \emph{comultiplication} is an algebra homomorphism $\Delta: A\longrightarrow M(A\otimes A)$
satisfying the following conditions
\begin{center}
	$\Delta(a)(1\otimes b)\in A\otimes A$\ \ \ \ \  and\ \ \ \ \  $(a\otimes 1)\Delta(b)\in A\otimes A$
\end{center}
and the co-associativity property 
\begin{center}
	$(a\otimes 1\otimes 1)((\Delta\otimes \imath)(\Delta(b)(1\otimes c)))=((\imath\otimes\Delta)((a\otimes 1)\Delta(b)))(1\otimes 1\otimes c),$
\end{center}
for all $a, b$,$c$ in $A$.

A pair $(A,\Delta)$ is called a \textit{multiplier Hopf algebra} if $\Delta$ is a comultiplication and the linear maps 
\begin{eqnarray*}
	T_1: A\otimes A &\longrightarrow &A\otimes A \quad \ \ \ \ \ \ \ \ \ \mbox{and} \ \ \ \ \ \ \ \ \ \quad  T_2:A\otimes A \longrightarrow A\otimes A\\
	\quad \ \ \ \ \ \ a\otimes b &\longmapsto &\Delta(a)(1\otimes b) \ \ \ \ \ \ \ \  \ \ \ \ \ \ \ \ \ \ \ \ \ \ \quad a\otimes b \longmapsto (a\otimes 1)\Delta(b)
\end{eqnarray*}
are bijective. Furthermore, a multiplier Hopf algebra $(A,\Delta)$ is called \emph{regular} if $(A,\sigma\Delta)$  is also a multiplier Hopf algebra, where $\sigma$ denotes the canonical flip map.

\vu

Analogously to the Hopf case, we have the existence of  unique linear maps called \emph{counit} and  \emph{antipode} for multiplier Hopf algebras.  The counit is an algebra homomorphism $\varepsilon :A\longrightarrow\Bbbk$ such that

\begin{center}
	$(\varepsilon\otimes \imath)(\Delta(a)(1\otimes b))=ab$ \ \ \ \ \ and \ \ \ \ \ $(\imath\otimes\varepsilon)((a\otimes 1)\Delta(b))=ab$
\end{center}
and the antipode is an algebra anti-homomorphism $S: A\longrightarrow M(A)$ such that
\begin{center}
	$m(S\otimes \imath)(\Delta(a)(1\otimes b))=\varepsilon(a)b$ \ \ \ \ \ and \ \ \ \ \ $m(\imath\otimes S)((a\otimes 1)\Delta(b))=\varepsilon(b)a,$
\end{center}
for all $a,b$ in $A$.

\vu

It is easy to check that any Hopf algebra is a multiplier Hopf algebra. Conversely, if $(A, \Delta)$ is a multiplier Hopf algebra and $A$ is unital, thus $A$ is a Hopf algebra. This shows that the notion of a multiplier Hopf algebra is a natural extension of a Hopf algebra for nonunital algebras.

\vu

The concept of multiplier Hopf algebras was motivated by the algebra $A_G$ with pointwise product of the complex functions with finite support  on any group $G$, i.e. functions that assume nonzero values for a finite set of elements of $G$. In this case, the multiplier algebra $M(A_G)$ consists of all complex functions on $G$ and $A_G \otimes A_G$ can be identified with the complex functions with finite support from $G\times G$ to $\mathbb{C}$. Furthermore, $A_G$ is a multiplier Hopf algebra with comultiplication  given by $\Delta(f)(p,q)=f(pq)$, and therefore, the counit and the antipode are given by $\varepsilon(f)=f(1_G)$ and $(S(f))(p)=f(p^{-1})$, for all $f\in A_G$ and $p,q\in G$, respectively.

\vu

In classical theory, the dual Hopf algebra is also a Hopf algebra if its dimension is finite. On the other hand, for the context of regular multiplier Hopf algebras, A. Van Dale introduced, in \cite {Frame}, a linear dual structure using a left integral. Given a regular multiplier Hopf algebra $(A, \Delta)$, a nonzero linear functional $\varphi$ on $A$ is called a left integral if $(\imath\otimes\varphi)\Delta(a)=\varphi(a)$, for all $a\in A$. It can define the dual algebra  $$\hat{A}=\{\varphi(\underline{\hspace{0.2cm}} a) ; a\in A\},$$ which is also a regular multiplier Hopf algebra for any dimension.

\vd

An important property for a multiplier Hopf algebra $(A, \Delta)$ is the existence of bilateral local units. This means that,  for any finite set of elements $a_1, \ldots , a_n$ of $A$ there exists an element $e\in A$ such that $e a_i = a_i = a_i e$, for all $1 \leq i \leq n$ (see \cite{Sweedler}). This  fact is used to justify  that $A^2=A$, which allowed to show, in \cite {Multiplier}, that the comultiplication $\Delta$ is a nondegenerate algebra homomorphism (cf.  \cite[Appendix]{Multiplier}). The existence of bilateral local units also allows us to use the Sweedler's notation in this context (see \cite{Action} and \cite{Sweedler}).

\vu

In \cite{Action}, the authors introduced the notion of module algebra for a multiplier Hopf algebra $(A,\Delta)$ acting on an algebra $R$ with a nondegenerate product. We call $R$ a \textit{left $A$-module algebra} if there exists a surjective linear  map $\triangleright:A\otimes R\longrightarrow R$ denoted by $\triangleright(a\otimes x)=a\triangleright x$, satisfying $a\triangleright (b\triangleright x)=ab\triangleright x$ and $a\triangleright xy=(a_{(1)}\triangleright x)(a_{(2)}\triangleright y)$. This notion is usual in the classical Hopf action context, however we need to be careful with respect to the last identity, since $\Delta(a)$ does not lie in $A\otimes A$, but in $M(A\otimes A)$. The Sweedler notation has sense because we can write $y=e\triangleright y$, for some $e\in A$, then 
$$a\triangleright xy=\mu_R (\Delta(a)(1\otimes e))(x\otimes y)=(a_{(1)}\triangleright x)(a_{(2)}\triangleright y),$$
where $\mu_R$ denote the product from $R\otimes R$ to $R.$

\vu

In this paper, we have as objective to extend the theorem of existence of enveloping actions, also called globalization theorem, to the context of partial actions of multiplier Hopf algebras on algebras with a nondegenerate product. In addition, the uniqueness of the enveloping actions will be also investigated.

\vu

We will work in Section 2 with the vector space $ Hom (A, R) = \{f: A \rightarrow R \ ; \ \mbox {f is linear} \} $ in the context of multiplier Hopf algebras. In the case of Hopf algebras, this structure is fundamental for the development of several concepts. The most basic is when we consider $R=\Bbbk$ and $A$ a coalgebra, in this way we construct a dual structure for $A$, denoted by $ Hom (A, \Bbbk) = A^*$, which will always be an algebra with the convolution product  $ (fg) (a) = (f \otimes g) \Delta(a)$. 

In addition, considering $A$ a bialgebra, $Hom (A, A)$ is essential for the definition of the antipode $S$ as the convolutive inverse of the identity map in $Hom (A, A).$ When $A$ is a Hopf algebra, we also have that the algebra $Hom (A, R)$ is a fundamental tool for the construction of  a globalization for a partial $A$-module algebra $R$, as presented in \cite{MunizandBatista}.

The theory of partial actions of multiplier Hopf algebras on not necessarily unital algebras was developed in \cite{Grasiela}, generalizing the theory constructed by S. Caenepeel and K. Jassen in \cite{Caenepeel}, and also the theory developed by A. Van Daele in \cite{Action}. From these concepts, in section 3 we will introduce the notion of globalization for a partial module algebra extending the theory proposed in \cite{MunizandBatista}. Besides that, considering $ R $ a left $s$-unital algebra, that is, for every $ x \in R $, $ x = Rx $, we establish an one-to-one correspondence between globalizable partial actions of any group $ G $ on $ R $ and partial actions of the dual algebra $\hat{A}_G$ on $R$.

\vu

Throughout this paper, vector spaces and algebras   will be all considered over a fixed field $\Bbbk$. The symbol $\otimes$ will always mean $\otimes_{\Bbbk}$. The pair $(A,\Delta)$ (or simply $A$) will always denote a regular multiplier Hopf algebra and $R$ an algebra with a nondegenerate product, unless others conditions are required.

\section{Convolution algebra structures on subspaces of $Hom(A,R)$}	

\quad
 Consider a regular multiplier Hopf algebra $A$  and an algebra $R$  with a nondegenerate product. We know that there not exists a natural algebra structure on $Hom(A,R)$ given by the product convolution
 \begin{equation}\label{confusao}
 (f g)(a)= \mu_{R}((f \otimes g)\Delta(a)),
 \end{equation}
 where $\mu_{R}$ denotes the product of $R$. This happens because the right hand side of the equality (\ref{confusao}) only make sense if $\Delta(a)$ is covered. In this case, the idea is consider a subspace of $Hom(A,R)$ on which we can induce an algebra structure similar to convolution product. For such a purpose we define the vector subspace
 \begin{eqnarray*}
Hom^r(A,R)=\{\sum_i f_i(\underline{\hspace{0.3cm}}a_i);  \ a_i \in A, f_i \in Hom(A,R)\}.
 \end{eqnarray*}

Since $A$ has bilateral local units any linear combination of elements
$f_i(\underline{\hspace{0.3cm}}a_i)$ can be written as $f(\underline{\hspace{0.3cm}}e)$, for some $f\in Hom(A,R)$ and $e\in A.$ In fact, given $f=\sum_i^n f_i(\underline{\hspace{0.3cm}}a_i)\in Hom(A,R)$
and $e\in A$ such that $ea_i=a_i=a_ie$, for each $i \in \{1, ..., n\}$, we have 
$$\begin{array}{rl}
	f(\underline{\hspace{0.3cm}}e)(b)= f(be)= \sum_i^n f_i(bea_i)= \sum_i^n f_i(ba_i)= \sum_i^n f_i(\underline{\hspace{0.3cm}}a_i)(b),
\end{array}$$
for all $b\in A$, i.e. $\sum_i^n f_i(\underline{\hspace{0.3cm}} a_i)=f(\underline{\hspace{0.3cm}} e)$.

\begin{thm} \label{hom} $Hom^r(A,R)$ is an algebra with product given by
	$$(fg)(c)=\mu_{R}((f \otimes g)\Delta(c)),$$
	for all $c\in A$ and $f, g\in Hom^r(A,R)$.
\end{thm}
\begin{proof}
To show that the product is well defined, we start verifying that the product does not depend on the writing of the elements in  $Hom^r(A,R).$ Indeed, consider $f, \ g \in Hom^r(A,R) $ such that $g=g_1(\underline{\hspace{0.3cm}}b_1)$ and $g=g_2(\underline{\hspace{0.3cm}}b_2),$ for some $g_1,g_2\in Hom(A,R)$ and $b_1,b_2\in A$. By the regularity of $A$, $(f\otimes \imath_A)\Delta(c)\in R\otimes A$, for all $c\in A$. Then 
$$(f\otimes g_1)(\Delta(c)(1\otimes b_1))=(f\otimes g_2)(\Delta(c)(1\otimes b_2)).$$
Applying the multiplication map $\mu_{R}$ on both sides of the above equation we can define the linear map $h:A\longrightarrow R$ 
$$h(c)=\mu_R((f\otimes g_1)(\Delta(c)(1\otimes b_1)))=\mu_R((f\otimes g_2)(\Delta(c)(1\otimes b_2)))= (fg)(c),$$
for all $c\in A.$ Analogously, considering $f, \ g \in Hom^r(A,R) $ such that $f=f_1(\underline{\hspace{0.3cm}}a_1)=f_2(\underline{\hspace{0.3cm}}a_2)$, for some $f_1,f_2\in Hom(A,R)$ and $a_1,a_2\in A$, the product $fg$ does not depend on the writing of the elements in  $Hom^r(A,R).$

Moreover, $h\in Hom^r(A,R)$ since writing $f=f_1(\underline{\hspace{0.3cm}}a_1)$, $g=g_1(\underline{\hspace{0.3cm}}b_1)$ and $a_1\otimes b_1=\sum_i^n \Delta(p_i)(1 \otimes q_i)$,

$$	\begin{array}{rl}
h(c)&=\mu_{R}((f\otimes g_1)(\Delta(c)(1\otimes b_1)))\\
&= \mu_{R}((f_1 \otimes g_1)(\Delta(c)(a_1 \otimes b_1)))\\
&=\mu_R((f_1\otimes g_1)(\sum_i^n\Delta(cp_i)(1\otimes q_i)))\\
&=\sum_i^n h_i(cp_i)\\
&= (\sum_i^n h_i(\underline{\hspace{0.3cm}}p_i))(c),
\end{array}$$ 
for all $c\in A$, where $h_i:A\longrightarrow R$ are the linear maps given by
$h_i(d)=\mu_R((f_1\otimes g_1)(\Delta(d)(1\otimes q_i)))$ for all $d\in A$ and $i \in \{1, ..., n\}.$ 
	
Note that using the Sweedler notation, we can rewrite $(f(\underline{\hspace{0.3cm}} a)g(\underline{\hspace{0.3cm}} b))(c) =f(c_{(1)} a)g(c_{(2)} b)$, for all $c \in A$.

	To show the associativity of $Hom^r(A,R)$, consider $f(\underline{\hspace{0.3cm}} a), g(\underline{\hspace{0.3cm}} b)$ and $h(\underline{\hspace{0.3cm}} c)\in Hom^r(A,R)$, then
	$$\begin{array}{rcl}
(f(\underline{\hspace{0.3cm}}a) (g(\underline{\hspace{0.3cm}}b)h(\underline{\hspace{0.3cm}}c)))(d)&=&f(d_{(1)}a) (g(\underline{\hspace{0.3cm}}b)h(\underline{\hspace{0.3cm}}c))(d_{(2)})\\
&=& f(d_{(1)}a)(g(d_{(2)(1)}b)h(d_{(2)(2)}c))\\
&=& (f(d_{(1)(1)}a)g(d_{(1)(2)}b))h(d_{(2)}c)\\
&=&(f(\underline{\hspace{0.3cm}}a)g(\underline{\hspace{0.3cm}}b))(d_{(1)}) h(d_{(2)}c)\\
&=&((f(\underline{\hspace{0.3cm}}a)g(\underline{\hspace{0.3cm}}b))h(\underline{\hspace{0.3cm}} c))(d),
\end{array}$$
for all $d\in A$, that is, $f(\underline{\hspace{0.3cm}}a) (g(\underline{\hspace{0.3cm}}b)h(\underline{\hspace{0.3cm}}c))=(f(\underline{\hspace{0.3cm}}a)g(\underline{\hspace{0.3cm}}b))h(\underline{\hspace{0.3cm}} c)$.
\end{proof}

\begin{cor}\label{Hom_naodegenerado}
	The algebra $Hom^r(A,R)$ has a nondegenerate product.
\end{cor}
\begin{proof}
	Consider $f(\underline{\hspace{0.3cm}}a)g(\underline{\hspace{0.3cm}}b)=0$, for all $g(\underline{\hspace{0.3cm}}b)\in Hom^r(A,R)$. Then for all $g\in Hom(A,R)$ and $b,c\in A$, 
$$f(c_{(1)}a)g(c_{(2)}b)=\mu_R((f\otimes g)(\Delta(c)(1\otimes b)(a\otimes 1)))=0.$$

Since $\Delta(A)(1\otimes A)=A\otimes A$, 
\begin{equation}\label{produto}
f(pa)g(q)=0,
\end{equation}
for all $p,q\in A$ and $g\in Hom(A,R)$.
Note that considering $R$ any left $A$-module unital ($A\triangleright R=R$) we have that for each $y\in R$ there exists $e_y\in A$ such that $e_y\triangleright y=y$. Thus, for each $y\in R$, we defined $g_y:A\longrightarrow R$ given by $g_y(q)=q\triangleright y$, for all $q\in A$.
Directly, $g_y\in Hom(A,R)$, for all $y\in R$, because $\triangleright$ is a linear map.

Therefore, for each $y\in R$, $f(pa)y=f(pa)g_y(e_y)\stackrel{(\ref{produto})}{=}0$ and since $R$ has a nondegenerate product, $f(\underline{\hspace{0.3cm}}a)(p)=f(pa)=0$, for all $p\in A$, i. e. $f(\underline{\hspace{0.3cm}}a)=0$.

Similarly, $g(\underline{\hspace{0.3cm}}b)=0$ if $f(\underline{\hspace{0.3cm}}a)g(\underline{\hspace{0.3cm}}b)=0$, for all $f(\underline{\hspace{0.3cm}}a)\in Hom^r(A,R)$
\end{proof}

Analogously, we define the algebra $ Hom^l(A,R)=\{\sum_i f_i(a_i \underline{\hspace{0.3cm}}); \ a_i \in A, f_i \in Hom(A,R)\}$ with a nondegenerate product. Note that these algebras are not  necessarily unital, since the map $f(\underline{\hspace{0.3cm}}b)(a)=\varepsilon(ab)1_R=f(b\underline{\hspace{0.3cm}})(a)$, where $\varepsilon(b)=1_{\Bbbk}$, $a\in A$ and $f \in Hom(A,R)$, would be the natural identity for such algebras, however R has no unit. Therefore, $Hom^r(A,M(R))$ and $Hom^l(A,M(R))$ are unital algebras.

For the next result we recall the concept of a nondegenerate module. Let $V$ a left $B$-module via $\triangleright$, where $V$ is a vector space and $B$ is an algebra with a nondegenerate product. We say that $\triangleright$ is nondegenerate if the following holds: $B\triangleright x=0$ if and only if $x=0$. Similarly, for right module and bimodules. It is important to note that if $V$ is a left $B$-module algebra, then $\triangleright$ is nondegenerate.

\begin{cor}\label{cor_bimodulo} The following statements hold:
\begin{enumerate}
	\item[(i)] $Hom(A,R)$ is a nondegenerate $(Hom^l(A,R), Hom^r(A,R))-$bimodule via  $$(f(a\underline{\hspace{0.3cm}}) \triangleright g)(c)=f(ac_{(1)})g(c_{(2)}) \  \ \mbox{and} \ \ (g \triangleleft f(\underline{\hspace{0.3cm}} a))(c)=g(c_{(1)})f(c_{(2)}a);$$ 
	\item[(ii)] $Hom(A,R)$ is a nondegenerate $Hom^r(A,R)-$bimodule via  $$(f(\underline{\hspace{0.3cm}}a) \triangleright g)(c)=f(c_{(1)}a)g(c_{(2)}) \  \ \mbox{and} \ \ (g \triangleleft f(\underline{\hspace{0.3cm}} a))(c)=g(c_{(1)})f(c_{(2)}a);$$ 
	\item[(iii)] $Hom(A,R)$ is a nondegenerate $Hom^l(A,R)-$bimodule via  $$(f(a\underline{\hspace{0.3cm}}) \triangleright g)(c)=f(ac_{(1)})g(c_{(2)}) \  \ \mbox{and} \ \ (g \triangleleft f(a\underline{\hspace{0.3cm}}))(c)=g(c_{(1)})f(ac_{(2)}),$$ 
	\end{enumerate}
for all $a,c \in A$ and $f,g \in Hom(A,R)$.
\end{cor}
\begin{proof}
(i) Note that these actions independent of the written of the elements in $Hom^r(A,R)$ and $Hom^l(A,R)$. Indeed, suppose that $f_1(\underline{\hspace{0.3cm}}a_1)=f_2(\underline{\hspace{0.3cm}}a_2)$, and consider $e \in A$ such that $ea_1=a_1$ and $ea_2=a_2$. Then, for each $c\in A$,
 $$(g \triangleleft f_1(\underline{\hspace{0.3cm}} a_1) )(c)=g(c_{(1)})f_1(c_{(2)}a_1)=g(c_{(1)})f_2(c_{(2)}a_2)=(g \triangleleft f_2(\underline{\hspace{0.3cm}} a_2))(c),$$ where $\Delta(c)(1 \otimes e)=c_{(1)} \otimes c_{(2)}$. Analogously for the left action.

Now consider $f(\underline{\hspace{0.3cm}}a), g(\underline{\hspace{0.3cm}}b)\in Hom^r(A,R)$ and $h\in Hom(A,R)$, then writing $a\otimes b=\sum\limits_i^n \Delta(p_i)(1 \otimes q_i)$ and using the Sweedler notation in Theorem \ref{hom}
\begin{eqnarray*}
	(h\triangleleft f(\underline{\hspace{0.3cm}} a)g(\underline{\hspace{0.3cm}} b))(c)&=& \sum_i g(c_{(1)})\mu_{R}(f\otimes g)(\Delta(c_{(2)}p_i)(1\otimes q_i))\\
	&=& g(c_{(1)})f(c_{(2)}a)g(c_{(3)}b)\\
&=&((h\triangleleft f(\underline{\hspace{0.3cm}} a))\triangleleft g(\underline{\hspace{0.3cm}} b))(c),
\end{eqnarray*}
for all $c\in A$. The proof that $\triangleleft$ is nondegenerate is similar to demonstration of the Corollary \ref{Hom_naodegenerado}. Similarly, $Hom(A,R)$ is a nondegenerate left  $Hom^l(A,R)$-module.

 Besides that, 
\begin{eqnarray*}
((f(a\underline{\hspace{0.3cm}})\triangleright h)\triangleleft g(\underline{\hspace{0.3cm}} b))(c)&=&(f(a\underline{\hspace{0.3cm}})\triangleright h)(c_{(1)})g(c_{(2)}b)\\
&=& f(ac_{(1)})h(c_{(2)})g(c_{(3)}b)\\
&=& f(ac_{(1)})(h\triangleleft g(\underline{\hspace{0.3cm}}b))(c_{(2)})\\
&=& (f(a\underline{\hspace{0.3cm}})\triangleright (h\triangleleft g(\underline{\hspace{0.3cm}} b)))(c),
\end{eqnarray*}
for all $f(a\underline{\hspace{0.3cm}})\in Hom^l(A,R)$, $g(\underline{\hspace{0.3cm}}b)\in Hom^r(A,R)$, $h\in Hom(A,R)$ and $c\in A$. Therefore, $Hom(A,R)$ is a nondegenerate $(Hom^l(A,R), Hom^r(A,R))-$bimodule. 

The itens (ii) and (iii) are analogously.
\end{proof}

To simplify the notation we denote by $\ast^r$ and $\ast^l$, the actions of $Hom^r(A,R)$ and $Hom^l(A,R)$ on $Hom(A,R)$, respectively.

Inspired in the classical theory of Hopf algebras and by the paper \cite{bialgebra}, we present another application for the algebra $Hom^r(A,R)$ and $Hom^l(A,R)$, precisely, we relate the antipode $S$ as a convolution inverse of an element in $Hom(A,A)$. Therefore, using Corollary  \ref{cor_bimodulo} we  introduce the following definition.

\begin{defi} Let $A$ be a regular multiplier Hopf algebra and $f,g\in Hom(A,A)$. We say that $f$ is a \textit{convolution inverse of $g$} if
	\begin{enumerate}
		\item[(i)] $f\ast^r g(\underline{\hspace{0.3cm}} a)=u_a\circ\varepsilon$;
		\item[(ii)] $g(b\underline{\hspace{0.3cm}})\ast^l f=u_b\circ\varepsilon$,
	\end{enumerate}
	where $u_a,u_b:\Bbbk \longrightarrow A$ are such that $u_a(1_{\Bbbk})=a$ and $u_b(1_{\Bbbk})=b$, for all $a,b\in A$.
	\label{def_inversaconvolutiva}
\end{defi}

\begin{pro} Let $A$ be a regular multiplier Hopf algebra. The linear map $S\in Hom(A,A)$ is the antipode of $A$ if only if $S$ is a convolution  inverse of the identity map of $A$, denoted by $\imath$.
\end{pro}	
\begin{proof}
	Consider $S\in Hom(A,A)$ the antipode of $A$, then 
	\begin{eqnarray*}
		(S \ast^r \imath(\underline{\hspace{0.3cm}}a))(c)&=&S(c_{(1)})c_{(2)}a\\
		&=&\varepsilon(c)a\\
		&=&(u_a\circ\varepsilon)(c),
	\end{eqnarray*}
for all $a,c\in A$.	Analogously, $(\imath(b\underline{\hspace{0.3cm}} )\ast^l S)(c)=(u_b\circ\varepsilon)(c)$, for all $b,c\in A$. Therefore, $S$ is a convolution inverse of the map $\imath$.

	Converselly, suppose that $S$ is a convolution  inverse of the map $\imath$, then 
	\begin{eqnarray*}
		\varepsilon(c)a&=&(u_a\circ\varepsilon)(c)\\
		&\stackrel{\ref{def_inversaconvolutiva}}{=}&(S\ast^r \imath(\underline{\hspace{0.3cm}} a))(c)\\
		&=& S(c_{(1)})c_{(2)}a\\
		&=& m(S\otimes \imath)(\Delta(c)(1\otimes a)),
	\end{eqnarray*}
	for all $a,c\in A$. Similarly, $m(\imath\otimes S)((b\otimes 1)\Delta(c))=\varepsilon(c)b$, for all $b,c\in A$. Therefore, the map $S$ is the antipode of the algebra $A$.
	\end{proof}

The next lemma give us a family of examples of module algebras that are important for the development of the theory presented in this work.

\begin{lem}\label{cutucahom}The algebra $Hom^r(A,R)$ is a left $A$-module algebra via
	\begin{eqnarray*}
		\triangleright :A\otimes Hom^r(A,R) &\longrightarrow &Hom^r(A,R)\\
		a\otimes f(\underline{\hspace{0.3cm}} b)&\longmapsto &a\triangleright f(\underline{\hspace{0.3cm}} b):=f(\underline{\hspace{0.3cm}} ab).
	\end{eqnarray*}
\end{lem}
\begin{proof} First of all, note that this action does not depend on the written on $Hom^r(A,R)$. Indeed, if $f_1(\underline{\hspace{0.3cm}} b_1)=f_2(\underline{\hspace{0.3cm}} b_2) \in Hom^r(A,R)$, then $f_1(\underline{\hspace{0.3cm}}ab_1)(c)=f_1(cab_1)=f_2(cab_2)=f_2(\underline{\hspace{0.3cm}} ab_2)(c)$, for all $a,c \in A$, i. e. $a\triangleright f_1(\underline{\hspace{0.3cm}} b_1)=a\triangleright f_2(\underline{\hspace{0.3cm}} b_2)$, for all $a\in A$. 

	Given $a,b\in A$, $f(\underline{\hspace{0.3cm}}c)\in Hom^r(A,R),$ 
		\begin{center}
		$a\triangleright(b\triangleright f(\underline{\hspace{0.3cm}}c))=a\triangleright f(\underline{\hspace{0.3cm}}bc)=f(\underline{\hspace{0.3cm}}abc)        =ab\triangleright f(\underline{\hspace{0.3cm}}c)$
	\end{center}		
	 and it is unital because $f(\underline{\hspace{0.3cm}}c)=e\triangleright f(\underline{\hspace{0.3cm}}c),$ where $e\in A$ such that $ec=c$.
	
	Besides that, if $a\in A$, $f(\underline{\hspace{0.3cm}}b)$, $g(\underline{\hspace{0.3cm}}c)\in Hom^r(A,R)$ and $b\otimes c=\sum\limits_i^n \Delta(p_i)(1 \otimes q_i)$,
$$	\begin{array}{rcl}
		(a\triangleright (f(\underline{\hspace{0.3cm}} b)g(\underline{\hspace{0.3cm}} c)))(d)
\vspace{0.08cm}
		&=&\mu_{R}[(f\otimes g)(\sum_i^n \Delta(dap_i)(1\otimes q_i))]\\
	\vspace{0.08cm}
		&=&\mu_{R}[(f\otimes g)(\Delta(da)(b\otimes c))]\\
\vspace{0.08cm}
	&=&f(d_{(1)}a_{(1)}b)g(d_{(2)}a_{(2)}c)\\
\vspace{0.08cm}
		&=&[f(\underline{\hspace{0.3cm}}a_{(1)}b)g(\underline{\hspace{0.3cm}}a_{(2)}c)](d)\\
\vspace{0.08cm}
		&=&[(a_{(1)}\triangleright f(\underline{\hspace{0.3cm}}b))(a_{(2)}\triangleright g(\underline{\hspace{0.3cm}}c))](d),
	\end{array}$$
	for all $d\in A$.
	Therefore, $Hom^r(A,R)$ is a left $A$-module algebra.
\end{proof}

\section{Globalization for Partial Module Algebra}

\quad Our goal in this section is to generalize the enveloping action theory presented in \cite{MunizandBatista} to the multiplier Hopf algebra context. We will always deal with left partial actions  because right partial actions are defined in a similar way. From now consider $A$ a regular multiplier Hopf algebra and $R$ an nondegenerated algebra unless another specific condition is required.

\subsection{Partial Action of Multiplier Hopf Algebras}
We started with the definition of partial action and its properties in the context of multiplier Hopf algebra. 

 \begin{defi}\label{def27} \cite{Grasiela} Let $A$ be a regular multiplier Hopf algebra and $R$ a nondegenerate algebra. A triple $(R, \cdot, \e)$ is a \emph{partial $A$-module algebra} if $\cdot$ is a linear map
 	\begin{eqnarray*}
 		\cdot: A\otimes R & \longrightarrow & R\\
 		a\otimes x & \longmapsto & a\cdot x
 	\end{eqnarray*}
 	and $\e$ is a linear map $\e: A\longrightarrow M(R)$ satisfying the following conditions, for all $a,b\in A$ and $x,y\in R$:
 	\begin{enumerate}
 		\item[(i)] $a\cdot (x(b\cdot y)) = (a_{(1)}\cdot x)(a_{(2)}b\cdot y)$;
 		
 		\vu
 		
 		\item[(ii)] $\mathfrak{e}(a)(b\cdot x)=a_{(1)}\cdot (S(a_{(2)})b\cdot x)$ and $\mathfrak{e}(A)R\subseteq A\cdot R$;
 		
 		\vu
 		
 		\item[(iii)] given $a_1,...,a_n\in A$ and  $x_1, ..., x_m\in R$ there exists $b\in A$ such that $a_ib=a_i=ba_i$ and $a_i\cdot x_j=a_i\cdot (b\cdot x_j)$, for all $1\leq i\leq n$ and $1\leq j\leq m$;
 		
 		\vu
 		
 		\item[(iv)] $A\cdot x=0$ if and only if $x=0$, that is, $\cdot$ is a \emph{nondegenerate action}.
 	\end{enumerate}
 	
 	Under these conditions, the map $\cdot$ is called a \emph{partial action} of $A$ on $R$, and we say that it is \emph{symmetric} if the following additional conditions also hold:
 	\begin{enumerate}
 		\item[(v)] $a\cdot ((b\cdot x)y)=(a_{(1)}b\cdot x)(a_{(2)}\cdot y)$;
 		
 		\vu
 		
 		\item[(vi)] $(b\cdot x)\mathfrak{e}(a)=a_{(2)}\cdot(S^{-1}(a_{(1)})b\cdot x)$;
 		
 		\vu
 		
 		\item[(vii)]$R \mathfrak{e}(A)\subseteq A \cdot R $,
 	\end{enumerate}
 	for all $x,y\in R$ and $a,b\in A$.
 \label{def_acaoparcialmulti}
 \end{defi}

\begin{pro} \cite{Grasiela}
If $A$ and $R$ are unital algebras, then the conditions above are equivalent to
	\begin{itemize}
	\item[(i)] $1_A\cdot x=x$;
	\item[(ii)]  $a\cdot(x(b\cdot y))= (a_{(1)}\cdot x)(a_{(2)}b\cdot y)$;
\item[(iii)] $a\cdot ((b\cdot x)y)=(a_{(1)}b\cdot x)(a_{(2)}\cdot y).$
\end{itemize}
\end{pro}
\begin{proof}
	Indeed, if $A$ and $R$ are unital algebras then Definition \ref{def_acaoparcialmulti} follows taking the linear map $\mathfrak{e}: A\longrightarrow M(R)=R$ given by $\mathfrak{e}(a)=a\cdot 1_R$, for all $a\in A$. Conversely, it is enough to check that $1_A\cdot x=x$, for all $x\in R$. To do this take $a, 1_A\in A$ and $x\in R$. By (iii) of Definition \ref{def_acaoparcialmulti} there exists an element $b\in A$ such that $ba=a=ab$, $b1_A=1_A=1_Ab$ and $a\cdot b\cdot x=a\cdot x$. However, $1_A$ is the identity element of $A$, hence $b=1_Ab=1_A$ and $a\cdot 1_A\cdot x=a\cdot x$. Then we have $a\cdot 1_A\cdot x=a\cdot x$, for all $a\in A$ and using (iv) of Definition \ref{def_acaoparcialmulti} we conclude $1_A\cdot x=x$, for all $x\in R$.
\end{proof}

It is immediate to check that if  $R$ is a partial $A$-module algebra, then $R$ is an $A$-module algebra if and only if $\mathfrak{e}(a)=\varepsilon(a)1_{M(R)}$, for all $a \in A$.

 \begin{pro}\cite{Grasiela}
	Assume that $R$ is an $A$-module algebra via a global action $\triangleright$ and let $L\subset R$ be a unital bilateral ideal of $R$ with identity element $1_L$. Then $L$ is a symmetric partial $A$-module algebra via $a \cdot x = 1_L(a \triangleright x)$ and $\e(a)=a\cdot 1_L$,
	for all $a\in A$ and $x\in L$.
	\label{pro_induzida}
	\end{pro}

\subsection{Induced Partial Module Algebra}

\quad Based on the construction of induced partial action via projections given by F. Castro and G. Quadros in \cite {Glauber}, we will make the analogue for partial module algebras in the context of multiplier Hopf algebras in order to construct a globalization for these structures.

\begin{defi}\label{defprojdealgebras}
	Let $L$ be a subalgebra of $R$. A linear map  $\pi: R\longrightarrow R$ is  called of projection over $L$, if $Im{\pi}=L$ and $\pi(x)=x$ for all $x\in L$. Besides that, if $\pi$ is multiplicative, we say that $\pi$ is an algebra projection.
\end{defi}
\begin{defi}\label{Aprojecao}
	Let $R$ be a left $A$-module algebra via $\triangleright$, $L$ be a nondegenerated subalgebra of $R$ and $\pi: R\longrightarrow R$ be an algebra projection over $L$. We say that the map $\pi$ is an \textit{$A$-projection} if
	\begin{eqnarray*}\pi(a\triangleright (x(b\triangleright y)))=\pi (a\triangleright (x\pi(b\triangleright y))),\end{eqnarray*}
	for all $x,y\in L$ and $a,b\in A.$
	Besides that, we say that $\pi$ is a \textit{symmetric $A$-projection}  if it also satisfies
	\begin{eqnarray*}
		\pi(a\triangleright ((b\triangleright x)y))=\pi (a\triangleright (\pi(b\triangleright x)y)),
	\end{eqnarray*}
	for all $x,y\in L$ and $a,b\in A.$
\end{defi}
\begin{exa}\label{projecao com idempotente}
	Let $R$ be a left $A$-module algebra via $\triangleright$ and $L$ be the subalgebra of $R$ generated by a central idempotent  $f\in R$. The map
	\begin{eqnarray*}
		\pi: R & \longrightarrow & R\\
		y & \longmapsto & fy
	\end{eqnarray*}
	is a symmetric $A$-projection.
\end{exa}

\begin{exa}
	Let $A$ be a left $A$-module algebra such that $a\triangleright b=a_{(1)}bS(a_{(2)})$, for all $a,b \in A$. Then any algebra projection of $A$ on a subalgebra $B$ with nondegenerate product is a symmetric $A$-projection.
\end{exa}

It is possible to construct a symmetric partial action from a global one via symmetric $A$-projections, as can be seen in the next result.

\begin{pro}\label{induzidaviapi}
	Let $R$ be a left $A$-module algebra via $\triangleright$ and $L$ be a nondegenerate subalgebra of $R$. If the linear map $\pi: R\longrightarrow R$ is a symmetric $A$-projection over $L$, then $L$ is a symmetric partial $A$-module algebra  via $a\cdot x=\pi(a\triangleright x)$ and $\e:A\longmapsto M(L)$ is such that
	\begin{eqnarray*}
		\overline{\e(a)}(b\cdot x)&=&\pi(a_{(1)}\triangleright\pi(S(a_{(2)})b\triangleright x)),\\
		\overline{\overline{\e(a)}}(b\cdot x)&=&\pi(a_{(2)}\triangleright\pi(S^{-1}(a_{(1)})b\triangleright x)),
	\end{eqnarray*}
	for all $a,b\in A$, $x\in L,$ where $\e(a)=(\overline{\e(a)},\overline{\overline{\e(a)}})\in M(L).$ 
\end{pro}
\begin{proof}Given $a,b\in A$, $x,y\in L,$
	\begin{eqnarray*}
		a\cdot (x(b\cdot y))&=& \pi(a\triangleright(x(\pi(b\triangleright y))))\\
		&\stackrel{\ref{Aprojecao}}{=}& \pi (a\triangleright(x(b\triangleright y)))\\
		&=& \pi(a_{(1)}\triangleright x)\pi(a_{(2)}b\triangleright y)\\
		&=& (a_{(1)}\cdot x)(a_{(2)}b\cdot y),	
	\end{eqnarray*}
and, analogously, $a\cdot ((b\cdot x)y)=(a_{(1)}b\cdot x)(a_{(2)}\cdot y)$. Thus the items (i) and (v) of Definition \ref{def_acaoparcialmulti} are verified. 

Note that the map $\e$ is well defined. Indeed,
\begin{eqnarray*}
	(b\cdot x)\overline{\e(a)}(c\cdot y)&=&\pi(b\triangleright x)\pi(a_{(1)}\triangleright\pi(S(a_{(2)})c\triangleright y))\\
	&=&\pi((b\triangleright x)(a_{(1)}\triangleright\pi(S(a_{(2)})c\triangleright y)))\\
	&=&\pi(a_{(2)}\triangleright((S^{-1}(a_{(1)})b\triangleright x)\pi(S(a_{(3)})c\triangleright y)))\\
	&\stackrel{\ref{Aprojecao}}{=}&\pi(a_{(2)}\triangleright (\pi(S^{-1}(a_{(1)})b\triangleright x)\pi(S(a_{(3)})c\triangleright y)))\\
	&\stackrel{\ref{Aprojecao}}{=}&\pi(a_{(2)}\triangleright (\pi(S^{-1}(a_{(1)})b\triangleright x)(S(a_{(3)})c\triangleright y)))\\
	&=&\pi(a_{(2)}\triangleright (\pi(S^{-1}(a_{(1)})b\triangleright x))(a_{(3)}S(a_{(4)})c\triangleright y))\\
	&=&\pi(a_{(2)}\triangleright \pi(S^{-1}(a_{(1)})b\triangleright x))\pi(c\triangleright y)\\
	&=&(\overline{\overline{\e(a)}}(b\cdot x))(c\cdot y),
\end{eqnarray*}
for all $b,c\in A$, $x,y\in L.$ Therefore, $\e(a)\in M(L)$, for all $a\in A$.

Besides that, $A\cdot L= L$ since given $x\in L$,
$x=\pi(x)=\pi(e\triangleright x)=e\cdot x$. Then the items (ii), (iv), (vi) and (vii) of Definition \ref{def_acaoparcialmulti} are satisfied immediately.

For the item (iii), given $a_1,...,a_n\in A$, $x_1, ..., x_m\in L$, we choose $b\in A$ such that $x_j=b\triangleright x_j$ and $a_ib=a_i=ba_i.$ So,
	\begin{eqnarray*}
		a_i\cdot x_j =\pi(a_i \triangleright x_j)
	=\pi(a_i\triangleright \pi(b\triangleright x_j))
		=a_i\cdot(b\cdot x_j),
	\end{eqnarray*}
	for all $1\leq i\leq n,$ $1\leq j\leq m.$ Thus $L$ is a symmetric partial $A$-module algebra.
\end{proof}

In this case, the symmetric partial action $\cdot$ is called induced partial action. Note that it is essential to have a symmetric $A$-projection for the map $\e$ to be well defined.

\begin{rem}It is necessary to call attention for the fact that the Proposition \ref{pro_induzida} is a particular case of Proposition \ref{induzidaviapi}.
	\end{rem}
 
The following example illustrates the Proposition \ref{induzidaviapi}.	

\begin{exa} \label{exemplo principal} Let $A_G$ be the algebra of the functions from $G$ to $\Bbbk$ with finite support with basis $\{\delta_p\}_{p \in G}$ over $\Bbbk$, where $\delta_p(g)=\delta_{p,g}$ (the Kronecker symbol), for all $g\in G$, and $R$ the group algebra $\Bbbk G$. Suppose that $R$ is an $A_G$-module algebra via $\delta_p \triangleright h = \delta_p(h)h$, for all $p,h \in G$. Consider a finite and normal subgroup $N\neq 1_G$ of $G$, with order $|N|$  not divisible by the characteristic of $\Bbbk$ and define the subalgebra $L=f_NR$ of $R$, where $f_N=\frac{1}{|N|}(\sum\limits_{n\in N}n)$ is a central idempotent element in $R$. Thus by Example \ref{projecao com idempotente} the linear map $\pi: R \longrightarrow R$ defined by $\pi(h)=f_Nh$, for all $h\in R$ is a symmetric $A_G$-projection. Therefore, $L=f_NR$ is a symmetric partial $A_G$-module algebra given by
	\begin{eqnarray*}
		\delta_p\cdot (f_Nh)&=&	\pi(\delta_p\triangleright(f_Nh))\\
		&=&	f_N(\delta_p\triangleright(f_Nh))\\
		&=& \left\{
		\begin{array}{rl}
			\frac{1}{|N|} f_Np, & \text{if}\quad ph^{-1}\in N\\
			0, \ \ \ \  & \text{otherwise},
		\end{array} \right.
	\end{eqnarray*}
	and $\mathfrak{e}(\delta_p)=\delta_p\cdot f_N=\frac{1}{|N|}f_N$. Notice that taking $h=1_G$ and $1_G\neq p\in N$, then 
	\begin{center}
		$\varepsilon(\delta_p)f_N=\delta_{p,1_G}f_N=0$,
	\end{center}
	i.e., $\mathfrak{e}(\delta_p)\neq \varepsilon(\delta_p)f_N$. Hence, the induced partial action is not global.
	\end{exa}

\begin{exa} Let $A_G$ be the algebra of the functions from $G$ to $\Bbbk$ with finite support, defined in Example \ref{exemplo principal} and $R=Hom^r(A_G,M(A_G))$. Then $R$ is an $A_G$-module algebra given by $a\triangleright f(\underline{\hspace{0.3cm}} b)=f(\underline{\hspace{0.3cm}} ab)$. Now, we define the map
	\begin{eqnarray*}
		\theta : A_G &\longrightarrow & M(A_G)\\
		\delta_p &\longmapsto & \theta(\delta_p): G \longrightarrow \Bbbk\\
		& &\ \ \ \  q \longmapsto \theta(\delta_p)(q):=  \left\{
		\begin{array}{rl}
			0, & \text{ } p=q\\
			1,  & \text{ otherwise},
		\end{array}  \right.
	\end{eqnarray*}
	thus  $\theta(\underline{\hspace{0.3cm}} \delta_e)$, where $e=1_G$, is a central idempotent in $R=Hom^r(A_G,M(A_G))$, because the product in $A_G$ is pointwise. Thus by Example \ref{projecao com idempotente} the linear map $\pi: R \longrightarrow R$ defined by $\pi(f(\underline{\hspace{0.3cm}} b))=\theta(\underline{\hspace{0.3cm}} \delta_e)f(\underline{\hspace{0.3cm}} b)$ is a symmetric $A_G$-projection.

	Therefore, $L=\theta(\underline{\hspace{0.3cm}} \delta_e)R$ is a partial $A_G$-module algebra via $a\cdot y=\pi(a\triangleright y)$ and $\e(a)=\theta(\underline{\hspace{0.3cm}} \delta_e)(a\triangleright\theta(\underline{\hspace{0.3cm}} \delta_e))$, for all $a\in A_G$ and $y\in L$.
	
	However, in this case, $\e(a)=\varepsilon(a)\theta(\underline{\hspace{0.3cm}} \ \delta_e)$, for any $a\in A_G$, what means that the action is global.
\end{exa}

\begin{exa}
	Let $A_G$ be the algebra defined in Example \ref{exemplo principal}. Consider $A_G$ the $A_G$-comodule algebra via $\Delta$. Then $\widehat{A_G}$ is an $A_G$-module algebra via
	$$
	\varphi(\underline{\hspace{0.3cm}} \delta_{g})\triangleright \delta_h=(\imath\otimes\varphi)(\Delta(\delta_h)(1\otimes \delta_g))
	= \delta_{hg^{-1}}\varphi(\delta_g),
	$$
where $\Delta(\delta_h)(1\otimes \delta_g)=\delta_{hg^{-1}}\otimes\delta_g$, for all $\delta_g, \delta_h\in\widehat{A_G}$.

Now consider any subgroup $N$ of $G$, thus $A_N$ is a subalgebra of $A_G$ with nondegenerate product. Also, define the algebra projection $\pi:A_G \longrightarrow A_G$ such that $\pi(\delta_g)=\delta_g$ if $g\in N$ and $\pi(\delta_g)=0$ otherwise. In this case, $\pi$ is a symmetric $\widehat{A_G}$-projection. Indeed, given $\delta_{g},\delta_h\in A_G$ and $\delta_p, \delta_q\in A_N$,
\begin{eqnarray*}
\pi(\varphi(\underline{\hspace{0.3cm}} \delta_{g})\triangleright(\delta_p(\varphi(\underline{\hspace{0.3cm}} \delta_{h})\triangleright\delta_q)))&=& \hspace{-0.5cm}\pi(\varphi(\underline{\hspace{0.3cm}} \delta_{g})\triangleright \delta_p\delta_{qh^{-1}}\varphi(\delta_h))\\
&\stackrel{p=qh^{-1}}{=}&\hspace{-0.5cm} \varphi(\delta_h)\pi(\varphi(\underline{\hspace{0.3cm}} \delta_{g})\triangleright \delta_p)\\
&=& \hspace{-0.5cm}\varphi(\delta_h)\varphi(\delta_g)\pi(\delta_{pg^{-1}})\\
&=&\hspace{-0.5cm}\left\{
\begin{array}{rl}
\hspace{-0.2cm}	\varphi(\delta_h)\varphi(\delta_g)\delta_{pg^{-1}}, & g\in N,  p=qh^{-1} (h\in N)\\
	0 \hspace{1cm},  & \text{ otherwise}
\end{array}  \right.
\end{eqnarray*}
and
\begin{eqnarray*}
\pi(\varphi(\underline{\hspace{0.3cm}} \delta_{g})\triangleright(\delta_p\pi(\varphi(\underline{\hspace{0.3cm}} \delta_{h})\triangleright\delta_q)))&=& \hspace{-0.3cm}\pi(\varphi(\underline{\hspace{0.3cm}} \delta_{g})\triangleright \delta_p\pi(\delta_{qh^{-1}}))\varphi(\delta_h)\\
&\stackrel{h\in N}{=}& \hspace{-0.3cm}\varphi(\delta_h)\pi(\varphi(\underline{\hspace{0.3cm}} \delta_{g})\triangleright \delta_p\delta_{qh^{-1}})\\
&=&\hspace{-0.3cm}\left\{
\begin{array}{rl}
\hspace{-0.2cm}	\varphi(\delta_h)\varphi(\delta_g)\delta_{pg^{-1}}, & g\in N,  p=qh^{-1} (h\in N)\\
	0 \hspace{1cm},  & \text{ otherwise}.
\end{array}  \right.
\end{eqnarray*}
Then $\pi(\varphi(\underline{\hspace{0.3cm}} \delta_{g})\triangleright(\delta_p(\varphi(\underline{\hspace{0.3cm}} \delta_{h})\triangleright\delta_q)))=\pi(\varphi(\underline{\hspace{0.3cm}} \delta_{g})\triangleright(\delta_p\pi(\varphi(\underline{\hspace{0.3cm}} \delta_{h})\triangleright\delta_q)))$. The symmetry follows in an analogous way, i. e. $\pi$ is a symmetric $\widehat{A_G}$-projection. Therefore, $A_N$ is a partial $\widehat{A_G}$-module algebra by Proposition \ref{induzidaviapi}. 

Note that writing $\Delta(\delta_h)(\delta_g\otimes 1)=\delta_g\otimes\delta_{g^{-1}h}$
\begin{eqnarray*}
\e(\varphi(\underline{\hspace{0.3cm}} \delta_{g}))(\varphi(\underline{\hspace{0.3cm}} \delta_{h})\cdot\delta_p)&=& \pi(\varphi(\underline{\hspace{0.3cm}} \delta_{g})\triangleright\pi(\delta_{ph^{-1}g})\varphi(\delta_{g^{-1}h}))
\end{eqnarray*}
and $\widehat{\varepsilon}(\varphi(\underline{\hspace{0.3cm}} \delta_{g}))(\varphi(\underline{\hspace{0.3cm}} \delta_{h})\cdot\delta_p)=\varphi(\delta_g)\varphi(\delta_h)\pi(\delta_{ph^{-1}}),$ 
then if $g\notin N$ and $h,p\in N$, we have that on the one hand $\widehat{\varepsilon}(\varphi(\underline{\hspace{0.3cm}} \delta_{g}))(\varphi(\underline{\hspace{0.3cm}} \delta_{h})\cdot\delta_p)=\delta_{ph^{-1}}$ and on the other hand  $\e(\varphi(\underline{\hspace{0.3cm}} \delta_{g}))(\varphi(\underline{\hspace{0.3cm}} \delta_{h})\cdot\delta_p)=0$. Therefore, the induced partial action is not global.
\end{exa}

\subsection{Globalization for Partial Module Algebras}

In this part of the work we present under which conditions we obtain an enveloping action for a symmetric partial action of a multiplier Hopf algebra on a algebra with nondegenerate product. We begin with the notion of enveloping action in the context of Hopf algebras.

\begin{defi}\textnormal{\cite{MunizandBatista}}\label{defglobalizacaoMAHopf}
	Let $A$ be a Hopf algebra and $L$ be a partial $A$-module algebra. An \textit{enveloping action}, or a \textit{globalization}, of $L$ is a  pair $(R,\theta)$ that satisfies the following conditions:
	\begin{enumerate}
		\item [(i)] $R$ is an $A$-module algebra;
		\item [(ii)] The linear map $\theta: L\longrightarrow R$ is an algebra  monomorphism;
		\item [(iii)] The subalgebra $\theta(L)$ is a right ideal of $R$;
		\item [(iv)] The partial action over $L$ is equivalent to the induced partial action over $\theta(L)$, i.e. $\theta(a\cdot x)=a\cdot\theta(x)=\theta(1_L)(a\triangleright\theta(x))$, for all $a\in A$ and $x\in L$;
		\item [(v)] $R=A\triangleright \theta(L).$
	\end{enumerate}
\end{defi}

Notice that if the partial action is symmetric then the subalgebra $\theta(L)$ in the item (iii) of the Definition \ref{defglobalizacaoMAHopf} is a bilateral ideal of $R$.

Extending this notion to the regular multiplier Hopf algebra context, we present the following definition.

\begin{defi}\label{defglobalizacaoMA}
Let $A$ be a regular multiplier Hopf algebra and $L$ be a symmetric partial $A$-module algebra. An \textit{enveloping action}, or \textit{globalization}, of $L$ is a triple $(R,\theta,\pi)$ that satisfies the following conditions:
	\begin{enumerate}
		\item [(i)] $R$ is an $A$-module algebra;
		\item [(ii)] The linear map $\theta: L\longrightarrow R$ is an algebra monomorphism;
		\item [(iii)] The subalgebra $\theta(L)$ is a bilateral ideal of $R$;
		\item [(iv)] The linear map $\pi: R\longrightarrow R$ is a symmetric $A$-projection over $\theta(L)$ such that the partial action of $L$ is equivalent (via $\theta$) to the induced partial action over $\theta(L)$ via $\pi$, i.e,  $\theta(a\cdot x)=a\cdot\theta(x)=\pi(a\triangleright\theta(x))$,
		for all $a\in A$, $x\in L$;
		\item [(v)] $R=A\triangleright \theta(L).$
	\end{enumerate}
\end{defi}

Now we present some results and observations for the construction of globalization for partial module algebras.

\begin{defi}\label{quaseunitaria} We say that a symmetric partial $A$-module algebra $L$ is \textit{quasi unitary} if, given a finite set of elements $x_1,...,x_n\in L$, there exists $b\in A$ such that $b\cdot x_i=x_i$ and $ab\cdot x_i=a\cdot x_i$, for all $a\in A$.
\end{defi}

\begin{exa}
	Let $R$ be an algebra with nondegenerate product, $A_G$ be the algebra of the functions with finite support from $G$ to $\Bbbk$ defined in Example \ref{exemplo principal} and $N$ a finite subgroup of $G$ such that $char(\Bbbk)\nmid|N|$. We define the linear map
	\begin{eqnarray*}
		\lambda: A_G & \longrightarrow & \Bbbk\\
		\delta_g & \longmapsto & \left\{
		\begin{array}{rl}
			\frac{1}{|N|}, & \text{ } g\in N,\\
			0, & \text{ otherwise. }
		\end{array} \right.
	\end{eqnarray*}
	
	Then $R$ is a quasi unitary symmetric partial $A_G$-module algebra via $\delta_g\cdot x=\lambda(\delta_g)x$, for all $\delta_g\in A_G$ and $x\in R$. It is enough to  consider $N=\{h_1,...,h_m\}$ a finite subgroup of $G$ and $b=\displaystyle\sum_{i=1}^m\delta_{h_i}\in A_G.$
\end{exa}

\begin{exa}	Every induced partial action via symmetric $A$-projection is quasi unitary. Indeed, let $R$ be a left $A$-module algebra via $\triangleright,$ $L$ a subalgebra of $R$ with nondegenerate product and $a\cdot x=\pi(a\triangleright x)$ the induced partial action via an symmetric $A$-projection $\pi,$ for all $a\in A$ and $x\in L$. Given $x_1, ..., x_n\in L$, we consider $b\in A$ such that $b\triangleright x_i=x_i$, for all $1\leq i\leq n.$ Thus we have
	\begin{eqnarray*}
		b\cdot x_i =\pi(b \triangleright x_i)=\pi(x_i)=x_i,
	\end{eqnarray*}
	and, $a\cdot x_i =\pi(a \triangleright x_i)=\pi(a \triangleright(b\triangleright x_i))=\pi(ab\triangleright x_i)=ab\cdot x_i,$
	for all $a\in A.$ Therefore, the induced partial action is quasi unitary.
	
\end{exa}

\begin{lem} \label{imersao phi}
Let $L$ be a quasi unitary symmetric partial $A$-module algebra and consider the linear map
\begin{eqnarray*}\label{fi}
\varphi: L & \longrightarrow & Hom(A,L)\\
x & \longmapsto & \varphi(x)(a):=f_x(a)=a \cdot x.
\end{eqnarray*} Then
	\begin{enumerate}
		\item[(i)] $\varphi(x)\in Hom^r(A,L)$, for all $x\in L$;
		\item [(ii)] $\varphi$ is an algebra monomorphism on $Hom^r(A,L)$;
		\item [(iii)] $\varphi(x)(a\triangleright \varphi(y))=\varphi(x(a\cdot y))$, for all $a\in A$ and $x,y\in L$;
		\item[(iv)] $(a\triangleright \varphi(x))\varphi(y)=\varphi((a\cdot x)y)$, for all $a\in A$ and $x,y\in L$,
\end{enumerate}
where the action $\triangleright$ of $A$ on $Hom^r(A,L)$ was defined in Lemma \ref{cutucahom}.
\end{lem}

\begin{proof} (i) By Definition \ref{quaseunitaria}, for each $x\in L$, there exists an element $b \in A$ such that $x=b\cdot x$ and $ab \cdot x=a \cdot x$, for all $a\in A$. Then, writing $\varphi(x)=f_x\in Hom(A,L)$, $$\varphi(x)(a)=f_x(a)=a \cdot x=ab \cdot x= f_x(ab)=f_x(\underline{\hspace{0.3cm}}b)(a),$$ for all $a \in A$, i. e. $\varphi(x)=f_x(\underline{\hspace{0.3cm}}b)\in Hom^r(A,L) $.
	
	Note that given $x\in L$, if there exists $b,c\in A$ such that $x=b\cdot x=c\cdot x$ and $a\cdot x=ab\cdot x=ac\cdot x$, then 
	\begin{center}
		$f_x(\underline{\hspace{0.3cm}}b)(a)=ab\cdot x=ac\cdot x=f_x(\underline{\hspace{0.3cm}}c)(a),$
	\end{center}
	for all $a\in A$. Therefore, $\varphi(x)$ does not depend on the choice of the elements in Definition \ref{quaseunitaria}.
	
(ii)	Choose $b\in A$ for the set $\{x,y\}$ by Definition \ref{quaseunitaria}. Then
	\begin{eqnarray*}
		\varphi(xy)(a)	&=&a\cdot xy\\
		&=&a\cdot (x(b\cdot y))\\
		&=&(a_{(1)}\cdot x)(a_{(2)}b\cdot y)\\
		&=&(a_{(1)}b\cdot x)(a_{(2)}b\cdot y)\\
		&=&(\varphi(x)\varphi(y))(a),	
	\end{eqnarray*}
	for all $a\in A.$ Thereby, $\varphi$ is an algebra homomorphism.
	
	Besides that, if $x\in L$ is such that $\varphi(x)=0$, then $0= \varphi(x)(a)=a\cdot x=ab\cdot x, $
for every $a\in A$. In particular, if $a\in A$ is such that  $ab=b$, thus $x=b\cdot x=ab\cdot x=0,$ i. e. $\varphi$ is an injective map.
	
	(iii) Let $b\in A$ for the set $\{x,y\}.$ Thus	
	\begin{eqnarray*}
		\varphi(y)(a\triangleright\varphi(x))(c)&=&f_y(\underline{\hspace{0.3cm}}b)(a\triangleright f_x(\underline{\hspace{0.3cm}}b))(c)\\
		&=&(c_{(1)}b\cdot y)(c_{(2)}(ab)\cdot x)\\
		&=&(c_{(1)}\cdot y)(c_{(2)}a\cdot x)\\
		&=&c\cdot(y(a\cdot x))\\
		&=&\varphi(y(a\cdot x))(c),
	\end{eqnarray*}
	for all $c\in A.$ Therefore, $\varphi(y)(a\triangleright\varphi(x))=\varphi(y(a\cdot x)),$ $x,y \in L$, $a\in A.$	

(iv) By the symmetry of the partial action and taking $b\in A$ for the set $\{x,y\}$,
\begin{eqnarray*}
	((a\triangleright\varphi(x))\varphi(y))(c)&=&(f_x(\underline{\hspace{0.3cm}}ab)f_y(\underline{\hspace{0.3cm}}b))(c)\\
	&=&f_x(c_{(1)}ab)f_y(c_{(2)}b)\\
	&=&(c_{(1)}a\cdot x)(c_{(2)}\cdot y)\\
	&=&c \cdot ((a\cdot x)y)\\
	&=&\varphi((a\cdot x)y)(c),
\end{eqnarray*}
for all $c\in A$. Then $(a\triangleright\varphi(x))\varphi(y)  =\varphi((a\cdot x)y)$, for all $a\in A$, $x,y\in L.$

\end{proof}

\begin{lem}\label{produtacaoglobal}
	Let $R$ be any $A$-module algebra, $a$, $b \in A$ and $x$, $y\in R$. Then
	$$ (a\triangleright x)(b\triangleright y)=a_{(1)}\triangleright (x(S(a_{(2)})b\triangleright y)).$$
\end{lem}

\begin{proof}	
It follows directly from the definition of (global) module algebra.
\end{proof}

\begin{lem}\label{imersao phi2}
	Let $L$ be a quasi unitary symmetric partial $A$-module algebra and $\varphi: L\longrightarrow Hom^r(A,L)$ be as in Lemma \ref{imersao phi}. Define  $R=A\triangleright \varphi(L)$. Then
	\begin{enumerate}
		\item [(i)] $R$ is an $A$-submodule algebra of $Hom^r(A,L)$;
		\item [(ii)] $\varphi(L)$ is a bilateral ideal of $R$.
	\end{enumerate}
\end{lem}

\begin{proof} (i) Clearly, $R$ is an $A$-submodule of $Hom^r(A,L)$ and
		\begin{eqnarray*}
		(a\triangleright \varphi(x))(b\triangleright \varphi(y)) &\stackrel{\ref{produtacaoglobal}}{=}& a_{(1)}\triangleright (\varphi(x) (S(a_{(2)})b\triangleright \varphi(y)))\\
		&\stackrel{\ref{imersao phi}}{=}& a_{(1)}\triangleright \varphi(x(S(a_{(2)})b\cdot y))\in R,
	\end{eqnarray*}
	for all  $(a\triangleright \varphi(x)), (b\triangleright \varphi(y)) \in R=A\triangleright \varphi(L)$, i. e. $R$  is also a subalgebra. 
	
	Besides that, $R=A\triangleright \varphi(L)$ has a nondegenerate product.  Indeed, consider $(a\triangleright\varphi(x))(b\triangleright\varphi(y))=0$, for all $(a\triangleright\varphi(x))\in R$. Then
	 $$c\triangleright((a\triangleright\varphi(x))(b\triangleright\varphi(y)))=(c_{(1)}a\triangleright\varphi(x))(c_{(2)}b\triangleright\varphi(y))=0,$$
	 	 for all $a,c\in A$ and $x\in L$. Since $\Delta(A)(A\otimes 1)=A\otimes A$, thus $(p\triangleright\varphi(x))(qb\triangleright\varphi(y))=0$, for all $p,q\in A$ and $x\in L$. 
	 
	 Note that, for each $x\in L$ there exists $e_x\in A$ such that $e_x\triangleright\varphi(x)=\varphi(x)$, then 
	 $$\varphi(x)(qb\triangleright\varphi(y))\stackrel{\ref{imersao phi}}{=}\varphi(x(qb\cdot y))=0,$$
	  for all $q\in A$ and $x\in L$. Since $\varphi$ is a monomorphism and $L$ has a nondegenerate product $qb\cdot y=0,$ for all $q\in A$.
	 
	 Besides that, by Definition \ref{quaseunitaria} 
	 $$qb\cdot y=qbe\cdot y=f_y(qbe)=(b\triangleright f_y(\underline{\hspace{0.3cm}}e))(q)=(b\triangleright\varphi(y))(q),$$
	 for all $q\in A$, thus $b\triangleright\varphi(y)=0$. Analogously, using Lemma \ref{imersao phi} item (iv) we proved that $a\triangleright\varphi(x)=0$ if $(a\triangleright\varphi(x))(b\triangleright\varphi(y))=0$, for all $(b\triangleright\varphi(y))\in R$. 
	 
	 Therefore, $R$ is an $A$-submodule algebra of $Hom^r(A,L)$.

	 (ii) Observe that $\varphi(L)\subseteq R$, because given $x\in L$, by the item (i) in Lemma  \ref{imersao phi}
	\begin{center}
		$\varphi(x)=f_x(\underline{\hspace{0.3cm}}b)=f_x(\underline{\hspace{0.3cm}}eb)=e\triangleright f_x(\underline{\hspace{0.3cm}}b)=e\triangleright \varphi(x),$
	\end{center}
where $eb=b.$ And, by the items (iii) and (iv) in Lemma \ref{imersao phi}, we obtain that $\varphi(L)$ is a bilateral ideal of $R$.
\end{proof}

\begin{lem}\label{lema Aproj} In the above conditions,
	there exists a unique symmetric $A$-projection over $\varphi(L)$ defined by
	\begin{eqnarray*}
		\pi': Hom^r(A,L) & \longrightarrow & Hom^r(A,L) \\
		f & \longmapsto & \pi'(f):=\varphi(f_1(a)),
	\end{eqnarray*}
	if $f=f_1(\underline{\hspace{0.3cm}} a)$.
\end{lem}
\begin{proof} $\bullet$ The linear map $\pi'$ is well defined. Indeed, consider $f=f_1(\underline{\hspace{0.3cm}} a_1)=f_2(\underline{\hspace{0.3cm}} a_2)$ and take $e\in A$ such that $ea_1=a_1$ and $ea_2=a_2$. Then $f_1(a_1)=f_1(ea_1)=f_2(ea_2)=f_2(a_2)$, i. e. $\varphi(f_1(a_1))=\varphi(f_2(a_2))$.
	
	\vu
	
	$\bullet$ $Im\pi'=\varphi(L),$ since given $\varphi(x)\in \varphi(L),$ we have that $\varphi(x)\stackrel{\ref{quaseunitaria}}{=}\varphi(b\cdot x)=\varphi(f_x(b))=\pi'(f_x(\underline{\hspace{0.3cm}}b))$.
	
	\vu
	
	$\bullet$ $\pi'(\varphi(x))=\varphi(x)$, for all $x\in L.$
	
	In fact, $\pi'(\varphi(x))\stackrel{\ref{imersao phi}(i)}{=}\pi'(f_x(\underline{\hspace{0.3cm}}b))=\varphi(f_x(b))=\varphi(b\cdot x)\stackrel{\ref{quaseunitaria}}{=}\varphi(x)$.
	
	\vu
	
	$\bullet$ Let $f(\underline{\hspace{0.3cm}}a),g(\underline{\hspace{0.3cm}}c)\in Hom^r(A,L)$, $a\otimes c= \sum_{i}^{n} \Delta(p_i)(1\otimes q_i)$ and $ep_i=p_i,$ for all $i\in \{1, \ldots, n\}$,
	$$\begin{array}{rcl}
		\pi'(f(\underline{\hspace{0.3cm}}a))\pi'(g(\underline{\hspace{0.3cm}}c))&=&\varphi(f(a)g(c))\\
		\vspace{0.08cm}
		&=&\varphi(\mu_L(f\otimes g)(a\otimes c))\\
				\vspace{0.08cm}
		&=&\varphi(\mu_L((f\otimes g)(\Delta(p_i)(1\otimes q_i))))\\
		\vspace{0.08cm}
		&=&\varphi(\mu_L((f\otimes g)(\Delta(e)(a\otimes c))))\\

		&=&\varphi((f(\underline{\hspace{0.3cm}}a)g(\underline{\hspace{0.3cm}}c))(e))\\
		\vspace{0.08cm}
		&\stackrel{\ref{hom}}{=}& \sum_{i}^{n} \varphi(h_i(\underline{\hspace{0.3cm}}p_i)(e))\\
				\vspace{0.08cm}
		&=& \sum_{i}^{n} \varphi(h_i(ep_i))\\
				\vspace{0.08cm}
		&=& \sum_{i}^{n}\varphi(h_i(p_i))\\
				\vspace{0.08cm}
		&=& \sum_{i}^{n} \pi'(h_i(\underline{\hspace{0.3cm}}p_i))\\
				\vspace{0.08cm}
		&=&\pi'(f(\underline{\hspace{0.3cm}}a)g(\underline{\hspace{0.3cm}}c)).
	\end{array} $$
	 Thus $\pi'$ is an algebra homomorphism.
	
	\vu
	
	$\bullet$ $\pi'$ is a symmetric $A$-projection.
	\begin{eqnarray*}
		\pi'(a\triangleright (\varphi(y)\pi'(c\triangleright\varphi(x))))&=&\pi'(a\triangleright (\varphi(y)\pi'(c\triangleright f_x(\underline{\hspace{0.3cm}}b))))\\
		&=&\pi'(a\triangleright (\varphi(y)\pi'(f_x(\underline{\hspace{0.3cm}}cb))))\\
		&=&\pi'(a\triangleright (\varphi(y)\varphi(f_x(cb))))\\
		&=&\pi'(a\triangleright (\varphi(y)\varphi(cb\cdot x)))\\
		&=&\pi'(a\triangleright (\varphi(y(c\cdot x))))\\
		&\stackrel{\ref{imersao phi}}{=}&\pi'(a\triangleright (\varphi(y)(c\triangleright\varphi(x)))),
	\end{eqnarray*}
	for all $a,c\in A$, $x,y\in L$ and $b\in A$ such that $b\cdot x=x$ and $cb\cdot x=c\cdot x,$ by Definition \ref{quaseunitaria}. Analogously,  	
		$$\pi'(a\triangleright (\pi'(c\triangleright\varphi(x))\varphi(y)))
=\pi'(a\triangleright ((c\triangleright\varphi(x))\varphi(y))),$$
	for all $a,c\in A$, $x,y \in L.$ Therefore, $\pi'$ is a symmetric $A$-projection over $\varphi(L).$
\end{proof}

%
%
%
%

\begin{thm}[Globalization] \label{teoGlobalizacaoMA} Let $L$ be a quasi unitary symmetric partial $A$-module algebra and consider the linear maps $\varphi$ as in Lemma \ref{imersao phi} and $\pi'$ as in Lemma \ref{lema Aproj}. Then $(R, \varphi,\pi)$ is a globalization of  $L$, where $R=A\triangleright\varphi(L)$ and $\pi=\pi'|_R.$
\end{thm}
\begin{proof}
	Notice that if we use Lemmas \ref{cutucahom}, \ref{imersao phi} and \ref{imersao phi2}, it is enough to verify the item (iv) of Definition \ref{defglobalizacaoMA}.
	
	In Lemma \ref{lema Aproj} we have shown the existence of an symmetric $A$-projection $\pi'$ over $\varphi(L)$ such that $\pi'(f(\underline{\hspace{0.3cm}}a)) = \varphi(f(a))$, for every $f(\underline{\hspace{0.3cm}}a)\in Hom^r(A,L) $. 
	
	Now, to show that the partial action over $L$ is equivalent (via $\varphi$) to the induced partial action on  $\varphi(L)$ via $\pi'$, by Definition \ref{quaseunitaria}, consider $b\in A$ such that $b\cdot x=x$ and $ab\cdot x=a\cdot x,$ for all $a\in A.$ Then,  $a\cdot \varphi(x) = \pi'(a\triangleright \varphi(x))
		=\pi'(f_x(\underline{\hspace{0.3cm}}ab))
		=\varphi(f_x(ab))
		=\varphi(a\cdot x)$.
Finally, for $\pi=\pi'|_{R}$ the result follows.
\end{proof}

The Globalization Theorem gives us the construction of an enveloping action, which we will call the \emph{standard enveloping action} $ (R, \varphi, \pi) $ of $ L $. 

As in the theory of Hopf algebras, \cite{MunizandBatista}, if we consider the globalization of induced partial action, it will not always coincide with the original global action.

Next, we will show that  the standard enveloping action is unique unless isomorphism, following the same ideas of the authors in \cite{MunizandBatista}. To do so, the next result will be crucial.

\begin{lem}\label{lema auxiliar} Let $(\overline{R},\theta,\overline{\pi})$ be another enveloping action of $L$. Then for all $a\in A$, $x,y\in L,$
	$$\theta(x)(a\triangleright\theta(y))=\theta(x(a\cdot y)).$$
\end{lem}
\begin{proof}
	It follows by Definition \ref{defglobalizacaoMA} that
	\begin{eqnarray*}
		\theta(x)(a\triangleright\theta(y))&=&\overline{\pi}(\theta(x)(a\triangleright\theta(y)))\\
		&=&\overline{\pi}(\theta(x))\overline{\pi}(a\triangleright\theta(y))\\
		&=&\theta(x)\theta(a\cdot y)\\
		&=&\theta(x(a\cdot y)),
	\end{eqnarray*}
	for all $a\in A$, $x,y\in L.$
\end{proof}

\begin{pro}\label{phiepimorf}Let $(\overline{R},\theta,\overline{\pi})$ be another enveloping action of $L$. Then the linear map
	\begin{eqnarray*}
		\Phi: (\overline{R},\theta,\overline{\pi}) & \longrightarrow & (R, \varphi,\pi) \\
		\sum_{i}a_i\triangleright\theta(x_i) & \longmapsto & \sum_{i}a_i\triangleright\varphi(x_i),
	\end{eqnarray*}
	is an epimorphism of $A$-module algebras.
\end{pro}
\begin{proof}
The map  $\Phi$ is well defined. Indeed, let $x\in R$ such that $x=\sum_{i}a_i\triangleright\theta(x_i)=0,$ we will show that $\Phi(x)=0$. Note that for all $a\in A$,
	$$
	\begin{array}{rl}
0 = \overline{\pi}(a\triangleright (\sum_i a_i\triangleright \theta(x_i)))= \overline{\pi}(\sum_i aa_i\triangleright \theta(x_i)) = \theta(\sum_i aa_i \cdot x_i),
	\end{array}$$
		i.e. $\sum_i aa_i\cdot x_i=0$, since $\theta$ is a monomorphism.
	
	Thus if we consider $b\in A$ such that $b\cdot x_i=x_i$ and $cb\cdot x_i=c\cdot x_i$, for all $c\in A$, we have	
$$	\begin{array}{rcl}
		\Phi(x)(a) &=& \Phi(\sum_i a_i\triangleright \theta(x_i))(a) \\
		&=& (\sum_i a_i\triangleright \varphi(b\cdot x_i))(a)\\
		&=& \sum_i f_{x_i}(aa_ib)\\
		&=& \sum_i aa_i\cdot x_i\\
		&=&0,
	\end{array}$$
for all $a\in A$, what implies that $\Phi(x)=0.$ By definition, $\Phi$ is an epimorphism of $A$-modules. And, for all $a,c\in A$ and $y,z\in L$	
	\begin{eqnarray*}
		\Phi((a\triangleright \theta(y))(c\triangleright \theta(z))) &\stackrel{\ref{produtacaoglobal}}{=}& \Phi(a_{(1)}\triangleright (\theta(y)(S(a_{(2)})c\triangleright \theta(z))))\\
		&\stackrel{\ref{lema auxiliar}}{=}& \Phi(a_{(1)}\triangleright \theta(y(S(a_{(2)})c\cdot z)))\\
		&=& a_{(1)}\triangleright \varphi(y(S(a_{(2)})c\cdot z))\\
		&\stackrel{\ref{imersao phi}}{=}& a_{(1)}\triangleright (\varphi(y)(S(a_{(2)})c\triangleright \varphi(z)))\\
		&\stackrel{\ref{produtacaoglobal}}{=}& (a\triangleright \varphi(y))(c\triangleright \varphi(z))\\
		&=& \Phi(a\triangleright \theta(y))\Phi(c\triangleright \theta(z)),
		\end{eqnarray*}
	then $\Phi$ is an epimorphism of $A$-module algebras.
\end{proof}

\begin{defi}
	An enveloping action $(R,\theta,\pi)$ is \textit{minimal} if, for every  $A$-submodule $M$ of $R$ that satisfies $\pi(M)=0$, it implies $M=0$. 
\end{defi}

\begin{pro}
	The standard enveloping action $(R, \varphi, \pi)$ is minimal.
\end{pro}
\begin{proof}
	It is enough to show the result for $M=\langle \sum_i a_i\triangleright \varphi(x_i) \rangle= \langle\sum_i Aa_i\triangleright \varphi(x_i)\rangle$, a cyclic $A$-submodule of $R$. Thus
$$	\begin{array}{rcl}
		\varphi(\sum_i (a_i \triangleright \varphi(x_i))(c))
		\vspace{0.08cm}
		 &\stackrel{\ref{imersao phi}}{=}&\varphi(\sum_i a_i\triangleright f_{x_i}(\underline{\hspace{0.3cm}}b)(c))\\
				\vspace{0.08cm}
		&=&\varphi(\sum_i f_{x_i}(ca_ib))\\
				\vspace{0.08cm}
		&=& \pi(\sum_if_{x_i}(\underline{\hspace{0.3cm}}ca_ib))\\
				\vspace{0.08cm}
				&=& \pi(\sum_i ca_i\triangleright \varphi(x_i))\\
				\vspace{0.08cm}
		&=& 0.
			\end{array}$$
	Since $\varphi$ is a monomorphism, it follows that	$\sum_i (a_i \triangleright \varphi(x_i))(c)=0$ for all $c\in A$, what implies $M=\langle \displaystyle\sum_i a_i\triangleright \varphi(x_i) \rangle =0.$		
\end{proof}

\begin{thm} Any two minimal enveloping actions are isomorphic as $A$-module algebras.
\end{thm}
\begin{proof}
	 We consider $(\overline{R}, \theta, \overline{\pi})$ a minimal enveloping action and $\Phi$ given by Lemma \ref{phiepimorf}. Let
	$x=\sum_i a_i\triangleright \theta (x_i)\in \overline{R}$
	such that $\Phi(x)=0.$ Then $0 = \sum_i (a_i\triangleright \varphi(x_i))(c)= \sum_i ca_i\cdot x_i,$ for all $c\in A.$ In particular, we can choose an element $c\in A$ such that $ca_i=a_i$, so
	$0 =\theta(\sum_i a_i \cdot x_i)= \overline{\pi}(\sum_i a_i \triangleright \theta(x_i)).$
	By the minimality of  $\overline{R}$ follows that $x=\sum_i a_i\triangleright \theta (x_i)=0,$ what means that $\Phi$ is injective. Therefore, by Lemma \ref{phiepimorf}, we obtain that $\Phi$ is an isomorphism of $A$-module algebras.
\end{proof}

\begin{rem} If $A$ and $L$ have identity $1_A$ and $1_L$, respectively, we have that $Hom^r(A,L)=Hom(A,L).$ In this case, we define	$\pi'(f)=\varphi(f(1_A))$
	and	it follows that
	\begin{eqnarray*}
		\pi'(a\triangleright\varphi(x)) &=& \varphi((a\triangleright \varphi(x))(1_A))\\
		&=&\varphi(f_x(\underline{\hspace{0.3cm}}a)(1_A))\\
		&=&\varphi(f_x(a))\\
		&=& \varphi(1_L(a\cdot x))\\
		&\stackrel{\ref{imersao phi}}{=}& \varphi(1_L) (a\triangleright \varphi(x)),
	\end{eqnarray*}
	then $\pi'(a\triangleright \varphi(x))= \varphi(1_L)(a\triangleright \varphi(x)),$
	for all $a\in A$ and $x\in L$.
	This means that in the case of symmetric partial actions Definitions \ref{defglobalizacaoMA} and \ref{defglobalizacaoMAHopf} coincide, and more, we obtain a generalization of the standard enveloping action,  constructed in \cite{MunizandBatista}, for the case that $A$ is a Hopf algebra and  $R$ an algebra with  identity.
\end{rem}

\subsection{Globalization for Partial Group Actions}

In this section we aim to establish an one-to-one correspondence between globalizable partial actions of groups on nonunital algebras R and symmetric partial actions of a multiplier Hopf algebra dual on R. We start mentioning some definitions about partial group actions introduced in \cite{exel2}.

\begin{defi}\label{defpartialgroupact}
	 Let $G$ be a group with identity element $1_G$ and let $R$ be a ring. A \textit{partial action} $\alpha$ of $G$ on $R$ is a collection of bilateral ideals $R_g\subseteq R$ and	ring isomorphisms $\alpha_g: R_{g^{-1}}\longrightarrow R_{g}$ such that:
	
	 \begin{enumerate}
	 	\item [(i)] $R_{1_G}= R$ and $\alpha_{1_G}$ is the identity map of $R$;
	\item [(ii)] $R_{(gh)^{-1}} \supseteq \alpha_h^{-1}(R_h \cap R_{g^{-1}})$;
	\item [(iii)] $\alpha_g\circ\alpha_h(x)= \alpha_{gh}(x)$ for each $x\in  \alpha_h^{-1}(R_h \cap R_{g^{-1}})$.
\end{enumerate}
The conditions (ii) and (iii) mean that the isomorphism $\alpha_{gh}$ is an extension of the isomorphism
$\alpha_g\circ\alpha_h$. Moreover, it is easily see that (ii) can be replaced by a “stronger looking”
	condition:
	\begin{enumerate}
	\item [(ii')] $\alpha_g(R_{g^{-1}}\cap R_h)=R_g\cap R_{gh}$, for all $g, h\in G$.
	\end{enumerate}
\end{defi}

\begin{defi} An action $\beta$ of a group $G$ on a ring $S$ is said to be a \textit{globalization},
or an \textit{enveloping action}, for the partial action $\alpha$ of $G$ on a ring $R$ if there exists a monomorphism $\varphi:R\longrightarrow S$ such that:
\begin{enumerate}
	\item[(i)] $\varphi(R)$ is an ideal in S;
\item[(ii)] $S =\sum_{g\in G}\beta_g(\varphi(R))$;
\item[(iii)] $\varphi(R_g)=\varphi (R)\cap\beta_g(\varphi(R))$;
\item[(iv)] $\varphi\circ\alpha_g =\beta_g\circ\varphi$ on $R_{g^{-1}}$.
\end{enumerate}
\end{defi}

Besides that the globalization $(\beta,S)$ of a partial action $(\alpha, R)$ is called unique if for any other globalization $(\beta',S')$ of $(\alpha,R)$ there exists an isomorphism of rings $\phi:S\longrightarrow S'$ such that $\beta'_{g}\circ \phi=\phi\circ\beta_g$, for every $g\in G$.
	
The next result extends the globalization theorem obtained in \cite{exel2} for partial actions on rings with unit.

\begin{thm}\textnormal{\cite{Micha}} Let $\alpha$ be a partial action of a group $G$ on a left $s$-unital ring $R$. Then $\alpha$ admits a globalization if and only if the following two conditions are satisfied:
	\begin{enumerate}
	\item[(i)] $R_g$ is a left $s$-unital ring for every $g\in G$;
	\item[(ii)] For each $g\in G$ and $x\in R$ there exists a multiplier $\gamma_g(x)\in M(R)$ such that
	$R\gamma_g(x)\subseteq R_g$ and $\gamma_g(x)$, restricted to $R_g$ as a right multiplier, is $\alpha_{g^{-1}}V_x\alpha_g$.
	\end{enumerate}
	Moreover, if such a globalization exists, it is unique, and the ring under the global action is left $s$-unital.
\end{thm}

Let $R$  be an $s$-unital algebra with nondegenarate product, $G$ an any group and $\hat{A}_G=\{\varphi(\underline{\hspace{0.2cm}}\delta_g); \delta_g\in A_G\}.$

\begin{thm}\label{bijside1} If $\alpha$ is a globalizable partial action of $G$ on $R$ such that:
	\begin{enumerate}
	\item[(1)] $R_g$ has nondegenerate product for every $g\in G$;
    \item[(2)]There exists $\sigma_g\in M(R)$, for every $g\in G$, such that
    \begin{enumerate}
    	\item [(i)] $\sigma_g$ is a central idempotent in $M(R);$
    	\item [(ii)] $\alpha_g(\sigma_{g^{-1}}\sigma_h)=\sigma_g\sigma_{gh}$, for all $h\in G;$
    	\item [(iii)] $\alpha_g(x)=\alpha_g(x)\sigma_g$ for all $x\in R_{g^{-1}}$;
    	\item [(iv)] $R\sigma_g\subseteq R_g,$
        \end{enumerate}
		\end{enumerate}
	then $R$ is a symmetric partial left $\hat{A}_G$-module algebra via $\varphi(\underline{\hspace{0.2cm}}\delta_g)\cdot x=\alpha_g(x\sigma_{g^{-1}})$ and $\e: \hat{A}_G \longrightarrow M(R)$ such that $\e(\varphi(\underline{\hspace{0.2cm}}\delta_g))=\sigma_{g}$, for all $g\in G$ and $x\in R$.
	\end{thm}
\begin{proof}
	We need to check the items of Definition \ref{def27}.\\
	$\bullet$ $\varphi(\underline{\hspace{0.2cm}}\delta_g)\cdot(x(\varphi(\underline{\hspace{0.2cm}}\delta_h)\cdot y))=(\varphi(\underline{\hspace{0.2cm}}\delta_g)_{(1)}\cdot x)(\varphi(\underline{\hspace{0.2cm}}\delta_g)_{(2)}\varphi(\underline{\hspace{0.2cm}}\delta_h)\cdot y).$
	
	Indeed,
	\begin{eqnarray*}
	\varphi(\underline{\hspace{0.2cm}}\delta_g)\cdot(x(\varphi(\underline{\hspace{0.2cm}}\delta_h)\cdot y))
	    &=&\alpha_g(x(\alpha_h(y\sigma_{h^{-1}}))\sigma_{g^{-1}})\\
		&\stackrel{(i)}{=}&\alpha_g(x\sigma_{g^{-1}})\alpha_g(\alpha_h(y\sigma_{h^{-1}})\sigma_{g^{-1}})\\
		&\stackrel{(iii)}{=}&\alpha_g(x\sigma_{g^{-1}})\alpha_g(\alpha_h(y\sigma_{h^{-1}})\sigma_h\sigma_{g^{-1}})\\
		&\stackrel{(ii)}{=}&\alpha_g(x\sigma_{g^{-1}})\alpha_g(\alpha_h(y\sigma_{h^{-1}})\alpha_h(\sigma_{h^{-1}}\sigma_{(gh)^{-1}}))\\
		&\stackrel{\ref{defpartialgroupact}}{=}&\alpha_g(x\sigma_{g^{-1}})\alpha_{gh}(y\sigma_{h^{-1}}\sigma_{(gh)^{-1}})\\
		&\stackrel{(i)}{=}&\alpha_g(x\sigma_{g^{-1}})\alpha_{gh}(y\sigma_{(gh)^{-1}})\alpha_{gh}(\sigma_{(gh)^{-1}}\sigma_{h^{-1}})\\
		&\stackrel{(ii)}{=}&\alpha_g(x\sigma_{g^{-1}})\alpha_{gh}(y\sigma_{(gh)^{-1}})\sigma_{(gh)}\sigma_{ghh^{-1}}\\
		&\stackrel{(iii)}{=}&\alpha_g(x\sigma_{g^{-1}})\alpha_{gh}(y\sigma_{(gh)^{-1}})\\
		&=&(\varphi(\underline{\hspace{0.2cm}}\delta_g)\cdot x)(\varphi(\underline{\hspace{0.2cm}}\delta_{gh})\cdot y)\\
		&=&(\varphi(\underline{\hspace{0.2cm}}\delta_g)_{(1)}\cdot x)(\varphi(\underline{\hspace{0.2cm}}\delta_g)_{(2)}\varphi(\underline{\hspace{0.2cm}}\delta_h)\cdot y),	
	\end{eqnarray*}
	for all $x,y\in R$, $\varphi(\underline{\hspace{0.2cm}}\delta_g),\varphi(\underline{\hspace{0.2cm}}\delta_h)\in \hat{A}_G.$
	
	$\bullet$ First we observe that $R=\hat{A}_G\cdot R$, since
$\varphi(\underline{\hspace{0.2cm}}\delta_{1_G})\cdot x=\alpha_{1_G}(x\sigma_{1_G})\stackrel{\ref{defpartialgroupact}}{=}\alpha_{1_G}(x)=x,$ for all $x\in R$. Then $\sigma_g(x)\in R=\hat{A}_G\cdot R$ since $\sigma_g\in M(R)$ and $x\in R$, that is, $\sigma_{g}R\subseteq\hat{A}_G\cdot R$. Besides that,
 	\begin{eqnarray*}
			\varphi(\underline{\hspace{0.2cm}}\delta_g)_{(1)}\cdot(\hat{S}(\varphi(\underline{\hspace{0.2cm}}\delta_g)_{(2)})\varphi(\underline{\hspace{0.2cm}}\delta_h)\cdot x)&=&\varphi(\underline{\hspace{0.2cm}}\delta_g)\cdot(\varphi(\underline{\hspace{0.2cm}}\delta_{g^{-1}h})\cdot x)\\
			&=&\alpha_g(\alpha_{g^{-1}h}(x\sigma_{(g^{-1}h)^{-1}})\sigma_{g^{-1}})\\
			&=&\alpha_h(x\sigma_{h^{-1}})\sigma_{h}\sigma_{g}\\
			&=&\alpha_h(x\sigma_{h^{-1}})\sigma_{g}\\
			&=&\e(\varphi(\underline{\hspace{0.2cm}}\delta_g))(\varphi(\underline{\hspace{0.2cm}}\delta_h)\cdot x),	
		\end{eqnarray*}
		for all $x\in R$ and $\varphi(\underline{\hspace{0.2cm}}\delta_g),\varphi(\underline{\hspace{0.2cm}}\delta_h)\in \hat{A}_G.$ Thus $\e(\varphi(\underline{\hspace{0.2cm}}\delta_g))(\varphi(\underline{\hspace{0.2cm}}\delta_h)\cdot x)=\varphi(\underline{\hspace{0.2cm}}\delta_g)_{(1)}\cdot(\hat{S}(\varphi(\underline{\hspace{0.2cm}}\delta_g)_{(2)})\varphi(\underline{\hspace{0.2cm}}\delta_h)\cdot x).$

		$\bullet$ Given $\varphi(\underline{\hspace{0.2cm}}\delta_{g_1}),\varphi(\underline{\hspace{0.2cm}}\delta_{g_2}),\cdots,\varphi(\underline{\hspace{0.2cm}}\delta_{g_n})\in\hat{A}_G$ and $x_1,\cdots x_m\in R$, we just need to choose the element $b=\varphi(\underline{\hspace{0.2cm}}\delta_{1_G})\in\hat{A}_G$ to obtain
		$\varphi(\underline{\hspace{0.2cm}}\delta_{g_i})\cdot x_j=\varphi(\underline{\hspace{0.2cm}}\delta_{g_i})\cdot (b\cdot x_j)$, for all $1\leq i\leq n$ and $1\leq j\leq m$.	
		
		$\bullet$ Since $\varphi(\underline{\hspace{0.2cm}}\delta_g)\cdot x=0$ for all $g\in G$, then $\varphi(\underline{\hspace{0.2cm}}\delta_{1_G})\cdot x=0$. Therefore, $x=0.$
		
		Besides that, similar to what was done, we prove that the
		partial action of $\hat{A}_G$ on $R$ is symmetric.			
		\end{proof}

\begin{rem}\label{nondegeneratealpha} Notice that for each $g\in G$, $\alpha_g$ is a nondegenerate algebra isomorphism, then it can be uniquely extended for $\alpha_g:M(R_{g^{-1}})\longrightarrow M(R_g).$ Besides that, $\sigma_g\sigma_h$ lies in $M(R_g)$.
	
\end{rem}

Conversely, we present the next result.
\begin{thm} If $(R,\cdot,\e)$ is a symmetric partial left $\hat{A}_G$-module algebra, then there exists a globalizable partial action $\alpha$ such that items \textnormal{(1)} and \textnormal{(2)} of Theorem \ref{bijside1} hold.	
\end{thm}
\begin{proof}
	We define, for all $g\in G$, $R_g=\e(\varphi(\underline{\hspace{0.2cm}}\delta_g))R$ and $\alpha_g: R_{g^{-1}} \longrightarrow  R_{g}$ such that $\alpha_g(x)=\varphi(\underline{\hspace{0.2cm}}\delta_g)\cdot x$, for all $x\in R_{g^{-1}} $.

Observe that $\alpha_g$ is well defined. Indeed, by the items (iii) and (iv) of Definition \ref{def27}, we have that $x=\varphi(\underline{\hspace{0.2cm}}\delta_{1_G})\cdot x$, then
\begin{eqnarray*}
	\varphi(\underline{\hspace{0.2cm}}\delta_g)\cdot x&=&\varphi(\underline{\hspace{0.2cm}}\delta_g)\cdot(\varphi(\underline{\hspace{0.2cm}}\delta_{1_G})\cdot x)\\
	&=&\varphi(\underline{\hspace{0.2cm}}\delta_g)\cdot(\varphi(\underline{\hspace{0.2cm}}\delta_{{g^{-1}}})\varphi(\underline{\hspace{0.2cm}}\delta_g)\cdot x)\\
	&=&\varphi(\underline{\hspace{0.2cm}}\delta_g)\cdot(\hat{S}(\varphi(\underline{\hspace{0.2cm}}\delta_g))\varphi(\underline{\hspace{0.2cm}}\delta_g)\cdot x)\\
	&\stackrel{\ref{def27}}{=}&\e(\varphi(\underline{\hspace{0.2cm}}\delta_g))(\varphi(\underline{\hspace{0.2cm}}\delta_g)\cdot x),	
\end{eqnarray*}
for all $x\in R$, $\varphi(\underline{\hspace{0.2cm}}\delta_g)\in \hat{A}_G.$ Thus $\varphi(\underline{\hspace{0.2cm}}\delta_g)\cdot x=\e(\varphi(\underline{\hspace{0.2cm}}\delta_g))(\varphi(\underline{\hspace{0.2cm}}\delta_g)\cdot x)\in R_g$. Besides that, $R_{1_G}=R$, since $x=\varphi(\underline{\hspace{0.2cm}}\delta_{1_G})\cdot x=\alpha_{1_G}(x)\in R_{1_G}$, for all $x\in R.$

Notice that $\e(\varphi(\underline{\hspace{0.2cm}}\delta_g))$ is a central idempotent element in $M(R).$ Indeed, for all $x\in R,$
\begin{eqnarray*}
\e(\varphi(\underline{\hspace{0.2cm}}\delta_g))\e(\varphi(\underline{\hspace{0.2cm}}\delta_g))x&=&\e(\varphi(\underline{\hspace{0.2cm}}\delta_g))\e(\varphi(\underline{\hspace{0.2cm}}\delta_g))(\varphi(\underline{\hspace{0.2cm}}\delta_{1_G})\cdot x)\\
&\stackrel{\ref{def27}}{=}&\varphi(\underline{\hspace{0.2cm}}\delta_g)\cdot (\varphi(\underline{\hspace{0.2cm}}\delta_{{g^{-1}}})\varphi(\underline{\hspace{0.2cm}}\delta_g)\cdot(\varphi(\underline{\hspace{0.2cm}}\delta_{{g^{-1}}})\cdot x))\\
&=&\varphi(\underline{\hspace{0.2cm}}\delta_g)\cdot (\varphi(\underline{\hspace{0.2cm}}\delta_{{g^{-1}}})\cdot x)\\
&=&\e(\varphi(\underline{\hspace{0.2cm}}\delta_g))x.
\end{eqnarray*}
 And,
\begin{eqnarray*}
	\e(\varphi(\underline{\hspace{0.2cm}}\delta_g))x
	 &\stackrel{\ref{def27}}{=}&\varphi(\underline{\hspace{0.2cm}}\delta_g)\cdot(\varphi(\underline{\hspace{0.2cm}}\delta_{{g^{-1}}})\varphi(\underline{\hspace{0.2cm}}\delta_{1_G})\cdot x)\\
	&\stackrel{\ref{def27}}{=}&x\e(\varphi(\underline{\hspace{0.2cm}}\delta_g)).
\end{eqnarray*}
Therefore, for all $m\in M(R)$,
\begin{eqnarray*}
	\e(\varphi(\underline{\hspace{0.2cm}}\delta_g))x&=&\e(\varphi(\underline{\hspace{0.2cm}}\delta_g))(mx)\\
	&=&(mx)\e(\varphi(\underline{\hspace{0.2cm}}\delta_g))\\
	&=&m(x\e(\varphi(\underline{\hspace{0.2cm}}\delta_g)))\\
	&=&m(\e(\varphi(\underline{\hspace{0.2cm}}\delta_g)x))\\
	&=&m\e(\varphi(\underline{\hspace{0.2cm}}\delta_g))x,
\end{eqnarray*}
for all $x\in R.$ Since $\e(\varphi(\underline{\hspace{0.2cm}}\delta_g))$ is a central element in $M(R)$, we have that $R_g$ is a right ideal of $R$  for each $g\in G.$	
\ \

$\bullet$ $\alpha_g$ is an isomorphism for each $g\in G$ and $x\in R$, since
$$\alpha_g(\alpha_{g^{-1}}(\e(\varphi(\underline{\hspace{0.2cm}}\delta_g))x))=\varphi(\underline{\hspace{0.2cm}}\delta_g)\cdot(\varphi(\underline{\hspace{0.2cm}}\delta_{g^{-1}})\cdot (\e(\varphi(\underline{\hspace{0.2cm}}\delta_g))x))	=\e(\varphi(\underline{\hspace{0.2cm}}\delta_g))x.$$

Analogously, we prove that $\alpha_{g^{-1}}(\alpha_g(\e(\varphi(\underline{\hspace{0.2cm}}\delta_{g^{-1}}))x))=\e(\varphi(\underline{\hspace{0.2cm}}\delta_{g^{-1}}))x.$ Besides that, for all $x,y\in R_{g{-1}}$,
 $$\alpha_g(xy)=
	\varphi(\underline{\hspace{0.2cm}}\delta_g)\cdot(x(\varphi(\underline{\hspace{0.2cm}}\delta_{1_G})\cdot y))
	\stackrel{\ref{def27}(i)}{=}\alpha_g(x)\alpha_g(y).$$
Therefore, $\alpha_{g}$ is an algebra isomorphism.
\ \

$\bullet$ $\alpha_{h^{-1}}(R_h\cap R_{g^{-1}})\subseteq R_{(gh)^{-1}}.$

First, we remark that for each $x\in R_g\cap R_h$ there exists $z\in R$ such that $x=\e(\varphi(\underline{\hspace{0.2cm}}\delta_g))\e(\varphi(\underline{\hspace{0.2cm}}\delta_h))z.$ Then given $x=\e(\varphi(\underline{\hspace{0.2cm}}\delta_h))\e(\varphi(\underline{\hspace{0.2cm}}\delta_{g^{-1}}))y\in R_h\cap R_{g^{-1}},$
\begin{eqnarray*}
	\alpha_{h^{-1}}(x)&=&\varphi(\underline{\hspace{0.2cm}}\delta_{h^{-1}})\cdot (\e(\varphi(\underline{\hspace{0.2cm}}\delta_h))\e(\varphi(\underline{\hspace{0.2cm}}\delta_{g^{-1}}))y)\\
	 &=&\varphi(\underline{\hspace{0.2cm}}\delta_{h^{-1}})\cdot(\varphi(\underline{\hspace{0.2cm}}\delta_h)\cdot(\varphi(\underline{\hspace{0.2cm}}\delta_{{h^{-1}}})\varphi(\underline{\hspace{0.2cm}}\delta_{g^{-1}})\cdot(\varphi(\underline{\hspace{0.2cm}}\delta_{g})\cdot y)))\\	
	&=&\varphi(\underline{\hspace{0.2cm}}\delta_{h^{-1}})\cdot (\varphi(\underline{\hspace{0.2cm}}\delta_h)\cdot (\varphi(\underline{\hspace{0.2cm}}\delta_{h^{-1}g^{-1}})\cdot (\varphi(\underline{\hspace{0.2cm}}\delta_g)\cdot y)))\\
	&=&\e(\varphi(\underline{\hspace{0.2cm}}\delta_{h^{-1}}))(\varphi(\underline{\hspace{0.2cm}}\delta_{(gh)^{-1}})\cdot(\varphi(\underline{\hspace{0.2cm}}\delta_{g})\cdot y))\\
	&=&\e(\varphi(\underline{\hspace{0.2cm}}\delta_{h^{-1}}))\e(\varphi(\underline{\hspace{0.2cm}}\delta_{(gh)^{-1}}))
	(\varphi(\underline{\hspace{0.2cm}}\delta_{(gh)^{-1}})\cdot(\varphi(\underline{\hspace{0.2cm}}\delta_g)\cdot y)),
\end{eqnarray*}
that is, $\alpha_{h^{-1}}(x)\in(R_{h^{-1}}\cap R_{(gh)^{-1}})\subseteq R_{(gh)^{-1}}.$

$\bullet$ $\alpha_{g}(\alpha_{h}(x))=\alpha_{gh}(x)$ for all $x\in \alpha_{h^{-1}}(R_h\cap R_{g^{-1}})$.
\begin{eqnarray*}
	\alpha_{g}(\alpha_{h}(x))
	&=&\varphi(\underline{\hspace{0.2cm}}\delta_{g})\cdot (\varphi(\underline{\hspace{0.2cm}}\delta_{h})\cdot x)\\
	&\stackrel{\ref{def27}}{=}&\e(\varphi(\underline{\hspace{0.2cm}}\delta_{g}))(\varphi(\underline{\hspace{0.2cm}}\delta_{gh})\cdot x)\\
	&=&\e(\varphi(\underline{\hspace{0.2cm}}\delta_{g}))\alpha_{gh}(x)\\
	&\stackrel{(\ast)}{=}&\alpha_{gh}(x),
\end{eqnarray*}
where in $(\ast)$ we use that $\alpha_{gh}(x)\in R_g$, since  $\alpha_{h^{-1}}(R_h\cap R_{g^{-1}})\subseteq R_{h^{-1}}\cap R_{(gh)^{-1}}.$ Thus $\alpha$ is a partial action of the group $G$ on $R.$

Now, we will prove that $\alpha$ is a globalizable action. First notice that $R_g$ is a left $s$-unital algebra for each $g\in G$, since $R$ is also left $s$-unital and $\e(\varphi(\underline{\hspace{0.2cm}}\delta_{g}))$ is a central idempotent in $M(R).$

$\bullet$ Define, for each $g\in G$ and $x\in R,$ $\gamma_g(x)=\varphi(\underline{\hspace{0.2cm}}\delta_{g})\cdot x.$ Naturally, $ R\gamma_g(x)\subseteq R_g.$ And,
\begin{eqnarray*} \alpha_{g}V_x\alpha_{g^{-1}}(\e(\varphi(\underline{\hspace{0.2cm}}\delta_{g}))y)
	&=&\varphi(\underline{\hspace{0.2cm}}\delta_{g})\cdot ((\varphi(\underline{\hspace{0.2cm}}\delta_{g^{-1}})\cdot (\e(\varphi(\underline{\hspace{0.2cm}}\delta_{g}))y))x)\\
	&=&(\varphi(\underline{\hspace{0.2cm}}\delta_{g})\varphi(\underline{\hspace{0.2cm}}\delta_{g^{-1}})\cdot (\e(\varphi(\underline{\hspace{0.2cm}}\delta_{g}))y))(\varphi(\underline{\hspace{0.2cm}}\delta_{g})\cdot x)\\
	&=&\e(\varphi(\underline{\hspace{0.2cm}}\delta_{g}))y(\varphi(\underline{\hspace{0.2cm}}\delta_{g})\cdot x)\\
	&=&\gamma(x)(\e(\varphi(\underline{\hspace{0.2cm}}\delta_{g}))y),
	\end{eqnarray*}
for all $y\in R$, i.e.  $\gamma_g(x)$ restricted to $R_g$ as a right multiplier, is $\alpha_{g}V_x\alpha_{g^{-1}}.$ Thus $\alpha$ is a globalizable action.

To finalize, we need to show the items (1) and (2) of Theorem \ref{bijside1}. The nondeneracy of the product in $R_g$ is immediate from the nondeneracy of the product in $R$.

Define $\sigma_g=\e(\varphi(\underline{\hspace{0.2cm}}\delta_{g}))\in M(R)$ for each $g\in G.$ By the previous observations, $\sigma_g$ is a central idempotent in $M(R)$. And it easy to check that	$\alpha_g(x)\sigma_{g}=\alpha_g(x)$, for all $x\in R_{g^{-1}}$ and $x\sigma_g\in R_g$ for all $x\in R.$

Besides that, we need to prove $\alpha_g(\sigma_{g^{-1}}\sigma_h)=\sigma_g\sigma_{gh},$ which is well defined by Remark \ref{nondegeneratealpha}. For such purpose, note that $\alpha_g(m)=(\varphi(\underline{\hspace{0.2cm}}\delta_{g})\cdot m)$, for all $m\in M(R_{g^{-1}}).$ Therefore,
\begin{eqnarray*}
\alpha_g(\sigma_{g^{-1}}\sigma_h)(x)&=&(\varphi(\underline{\hspace{0.2cm}}\delta_{g})\cdot(\e(\varphi(\underline{\hspace{0.2cm}}\delta_{g^{-1}}))\e(\varphi(\underline{\hspace{0.2cm}}\delta_{h}))))(\varphi(\underline{\hspace{0.2cm}}\delta_{1_G})\cdot x)\\
&=&\varphi(\underline{\hspace{0.2cm}}\delta_{g})\cdot(\e(\varphi(\underline{\hspace{0.2cm}}\delta_{g^{-1}}))\e(\varphi(\underline{\hspace{0.2cm}}\delta_{h}))(\varphi(\underline{\hspace{0.2cm}}\delta_{g^{-1}})\cdot x))\\
&=&\varphi(\underline{\hspace{0.2cm}}\delta_{g})\cdot(\e(\varphi(\underline{\hspace{0.2cm}}\delta_{h}))(\varphi(\underline{\hspace{0.2cm}}\delta_{g^{-1}})\cdot x))\\
&\stackrel{\ref{def27}}{=}&\varphi(\underline{\hspace{0.2cm}}\delta_{g})\cdot(\varphi(\underline{\hspace{0.2cm}}\delta_{h})\cdot (\varphi(\underline{\hspace{0.2cm}}\delta_{h^{-1}g^{-1}})\cdot x))\\
&\stackrel{\ref{def27}}{=}&\e(\varphi(\underline{\hspace{0.2cm}}\delta_{g}))(\varphi(\underline{\hspace{0.2cm}}\delta_{gh})\cdot (\varphi(\underline{\hspace{0.2cm}}\delta_{(gh)^{-1}})\cdot x))\\
&=&\e(\varphi(\underline{\hspace{0.2cm}}\delta_{g}))\e(\varphi(\underline{\hspace{0.2cm}}\delta_{gh})) x\\
&=&\sigma_g\sigma_{gh}(x),
\end{eqnarray*}
for all $x\in R_g$.
\end{proof}

\section{Acknowledgments}
We would like to thank to E. Batista for his availability and interest in
discussing our questions contributing to the evolution of this research. In the same
way to A. Van Daele for remaining willing to enlighten us about the theory of Multiplier
Hopf Algebras. Also to A. Paques whose corrections and suggestions
were very fruitful for the development of this work.
\addcontentsline{toc}{chapter}{Referências Bibliográficas}

\end{document}